\documentclass[
final
]{dmtcs-episciences}

\usepackage{amsthm}
\usepackage{amssymb}
\usepackage{multirow}
\usepackage[dvipsnames]{xcolor}
\usepackage{romannum}
\usepackage{mathtools}
\usepackage[linesnumbered,ruled]{algorithm2e}
\usepackage{tikz}
\usetikzlibrary{shapes, backgrounds, decorations.pathreplacing, patterns}

\newtheorem{thm}{Theorem}[section]
\newtheorem{prop}[thm]{Proposition}
\newtheorem{lem}[thm]{Lemma}
\newtheorem{cor}[thm]{Corollary}
\newtheorem{Observation}{Observation}
\newtheorem{problem}{Problem}[section]
\theoremstyle{definition}
\newtheorem{defn}[thm]{Definition} 
\newtheorem{eg}[thm]{Example}

\tikzstyle{1} = [diamond,  fill = mygray,  inner sep=7.2pt]
\tikzstyle{2} = [diamond,  fill = mygray,  inner sep=14pt]
\definecolor{mygray}{gray}{0.8}

\usepackage[round]{natbib}

\author[Aaron Slobodin, Gary MacGillivray, and Wendy Myrvold]{
  Aaron Slobodin\affiliationmark{1}\thanks{Supported by funding from Natural Sciences and Engineering Research of Council of Canada (NSERC).}
  \and Gary MacGillivray\affiliationmark{1}\thanks{Supported by NSERC.}
  \and Wendy Myrvold\affiliationmark{2}\thanks{Supported by NSERC.}
}
\title{Proving exact values for the $2$-limited broadcast domination number on grid graphs}
\affiliation{
  Department of Mathematics and Statistics, University of Victoria, Victoria, B.C., Canada\\
  Department of Computer Science, University of Victoria, Victoria, B.C., Canada
}
\begin{document}
\pagenumbering{arabic}
\publicationdata
{vol. 25:2 }
{2023}
{21}
{10.46298/dmtcs.11478}
{2023-06-16; 2023-06-16; 2023-10-05}
{2023-10-06}
\maketitle
\begin{abstract}
We establish exact values for the $2$-limited broadcast domination number of various grid graphs, in particular $C_m\square C_n$ for $3 \leq m \leq 6$ and all $n\geq m$, $P_m \square C_3$ for all $m \geq 3$, and $P_m \square C_n$ for $4\leq m \leq 5$ and all $n \geq m$. We also produce periodically optimal values for $P_m \square C_4$ and $P_m \square C_6$ for $m \geq 3$, $P_4 \square P_n$ for $n \geq 4$, and $P_5 \square P_n$ for $n \geq 5$. Our method completes an exhaustive case analysis and eliminates cases by combining tools from linear programming with various mathematical proof techniques. 
\end{abstract}

\section{Introduction}\label{intro}
Suppose there is a transmitter located at each vertex of a graph $G$. A \textit{$k$-limited broadcast} $f$ on $G$ is a function $f: V(G) \mapsto \left\{ 0,1,\dots, k \right\} $. The integer $f(v)$ represents the \textit{strength} of the broadcast from $v$, where $f(v)=0$ means the transmitter at $v$ is not broadcasting. A broadcast of positive strength $f(v)$ from $v$ is \textit{heard} by all vertices $u$ such that $d(u,v) \leq f(v)$, where $d(u,v)$ is the distance between the $u$ and $v$ in $G$. A broadcast $f$ is \textit{dominating} if each vertex of $G$ hears the broadcast from some vertex. The \textit{cost} of a broadcast $f$ is $\sum_{v \in V(G)} f(v)$. The \textit{$k$-limited broadcast domination number} $\gamma_{b,k}(G)$ of a graph $G$ is the minimum cost of a $k$-limited dominating broadcast on $G$.

The $k$-limited broadcast domination number can be seen to be the optimum solution to \ref{ILP} shown below. Let $G$ be a graph and fix $1 \leq k \leq rad(G)$, where $rad(G)$ is the \textit{radius} of $G$. For each vertex $i \in V(G)$ and $\ell \in \left\{ 1,2,\dots, k \right\} $, let $x_{i,\ell} = 1$ if vertex $i$ is broadcasting at strength $\ell$ and $0$ otherwise.  
\renewcommand{\theequation}{ILP 1.1}
\begin{equation}\label{ILP}
\begin{aligned}
\textnormal{Minimize:} \quad &\sum_{\ell=1}^{k} \sum_{i \in V(G)} \ell\cdot x_{i,\ell}\\
\textnormal{Subject to:} \quad &(1) \quad\sum_{\ell=1}^{k} \sum_{
\substack{
\text{$i \in V(G) \textnormal{ s.t.}$}\\
\text{ $d(i,j) \leq \ell$}}
}
 x_{i,\ell} \geq 1, \quad \textnormal{for each vertex } j \in V(G),\\
&(2) \quad  x_{i,\ell} \in \left\{ 0,1 \right\}  \quad \textnormal{for each vertex } i \in V(G) \textnormal{ and } \ell \in \left\{ 1,2,\dots,k \right\}.\\
\end{aligned}
\end{equation}
\renewcommand{\theequation}{\arabic{equation}}

The $k$-limited broadcast domination number of a graph was first defined in \cite{ogerwin} (also see \cite{erwin}). The first major results for $k$-limited broadcast domination were given in 2018 and are specific to $2$-limited dominating broadcasts in trees \cite{thing}. These results were generalized to $k$-limited dominating broadcasts to give a best possible upper bound of $ \gamma_{b,k} (T) \leq \left\lceil \frac{ k+2 }{ k+1 }\cdot \frac{ n }{ 3 } \right\rceil $, where $T$ is a tree on $n$ vertices \cite{thing2}. Specific to $2$-limited dominating broadcasts, if $G$ is a connected graph of order $n$, then $ \gamma_{b,2}(G)\leq \left\lceil \frac{ 4n }{ 9 }  \right\rceil $ and if $G$ is a graph of order $n$ that contains a dominating path, then $ \gamma_{b,2}(G)\leq \left\lceil \frac{ 2n }{ 5 }  \right\rceil $ \cite{thing}. Yang showed that if $G$ is a cubic $(C_4, C_6)-$free graph of order $n$, then $ \gamma_{b,2} (G) \leq \frac{ n }{ 3 } $ \cite{thesis, Henning2021}. Park recently extended this result to cubic $C_4-$free graph of order $n$ \cite{PARK2023178}.

For each fixed positive integer $k$, the problem of deciding whether there exists a $k$-limited dominating broadcast of cost at most a given integer $B$ is NP-complete \cite{thing2, thesis}. The results of \cite{thesis}, respectively, establish $O(n^3),$ $O(n^2),$ $O(n^2),$ and $O(n^3)$ time algorithms for the $k$-limited broadcast domination number of strongly chordal graphs, interval graphs, circular arc graphs, and proper interval bigraphs. The algorithm for $k$-limited broadcast on strongly chordal graphs in \cite{thesis} is a specialization of the $O(n^3)$ time algorithm  for (general) broadcast domination on strongly chordal graphs in \cite{frankmasters, frankrickandgarry}.

The $k$-limited broadcast domination problem is a restriction of the broadcast domination problem in which vertices can broadcast with strength up to $rad(G)$.  The \textit{broadcast domination number} $\gamma_b(G)$ of a graph $G$ is optimum solution to the ILP obtained from \ref{ILP} by setting $k = rad(G)$. The broadcast domination number of a graph was introduced in \cite{ogerwin}. Erwin proved that, for every non-trivial connected graph,
\begin{equation*}
\begin{aligned}
\left\lceil \frac{ diam(G) +1}{3 }  \right\rceil \leq   \gamma_b(G) \leq \textnormal{ min} \left\{ rad(G), \gamma(G) \right\}.
\end{aligned}
\end{equation*}
It immediately follows that $ \gamma_b(P_n) = \left\lceil \frac{ n }{ 3 }  \right\rceil $. Broadcast domination in trees was first explored in \cite{herkmasters} (also see \cite{COCKAYNE20111235, HERKE20095950}). This work establishes $ \gamma_b(T) \leq \left\lceil \frac{ n }{ 3 }  \right\rceil $, where $T$ is a tree of order $n$. The broadcast domination number is known for the Cartesian products of two paths \cite{DUNBAR200659}, two cycles \cite{DUNBAR200659}, and strong grids \cite{grids}. Note that, as $ \gamma_{b} \left( C_{m} \square C_{n} \right) \leq \gamma_{b} \left( P_{m} \square C_{n} \right) \leq \gamma_{b} \left( P_{m} \square P_{n} \right) $, the previously stated results provide bounds for $ \gamma_{b} \left( P_{m} \square C_{n} \right) $. A survey of results on broadcast domination can be found in \cite{Henning2021}.

This paper presents lower bounds for the $2$-limited broadcast domination number of the Cartesian products of two paths, a path and a cycle, and two cycles. Our computational approach completes an exhaustive search of all possible small induced sub-broadcasts of given costs on a graph. Cases which provably cannot be part of an optimal broadcast are then eliminated. This approach can likely be extended to other graphs as well as general $k$-limited broadcast domination.

Some intuition for our method is provided by example in Section \ref{sec:High Level Overview}. Section \ref{sec:backtrack} describes and proves the correctness of the six schemes used to eliminate cases in the exhaustive search. Section \ref{sec:backtrack} concludes with the statement of our main algorithm (Algorithm \ref{alg:thefinalfinal}) to prove lower bounds and the proof of its correctness. Our results are summarized in Section \ref{sec:bactrack:app of methods}. These include exact values for the $2$-limited broadcast domination number of $C_m\square C_n$ for $3 \leq m \leq 6$ and all $n\geq m$, $P_m \square C_3$ for all $m \geq 3$, and $P_m \square C_n$ for $4\leq m \leq 5$ and all $n \geq m$, and periodically optimal values for $P_m \square C_4$ and $P_m \square C_6$ for $m \geq 3$, $P_4 \square P_n$ for $n \geq 4$, and $P_5 \square P_n$ for $n \geq 5$. These results improve upon the bounds in Slobodin's M.Sc. thesis \cite{slobodin}.

\section{Intuition and Definitions}\label{sec:High Level Overview}
This section includes a high-level overview of our method, an example specific to $P_5 \square C_n$, and relevant definitions. 

\begin{defn}
Let $f$ be a $2$-limited broadcast on the graph $G$ and let $X \subseteq V(G)$. Define the \textit{sub-broadcast $g$ induced by $X$} by
\begin{equation*}
\begin{aligned}
g(x) &= \begin{cases}
f(x) & \textnormal{ if } x \in X \textnormal{ and }\\
0 & \textnormal{ otherwise}.
\end{cases}
\end{aligned}
\end{equation*}
\end{defn}

Throughout this paper, we consider a class of graphs $G_{m,n}$ equal to $P_m \square C_n$ or $C_m \square C_n$ for a fixed number of rows $m$. In this way, the vertex in the $i$th row and $j$th column can be denoted by $(i,j)$. The goal is to prove that $\gamma_{b,2}(G_{m,n})$ is greater than or equal to a function $B(m, n)$. We proceed by induction on $n$. After checking the appropriate base cases computationally, we assume the bound holds for all $n < n_0$ for some integer $n_0$. Let $f$ be a $2$-limited dominating broadcast of $G_{m,n_0}$. We choose values $r$ and $s$ such that, if the minimum cost (with respect to $f$) of a sub-broadcast induced by $r$ consecutive columns of $G_{m,n_0}$ is strictly greater than $s$, then $B(m, n_0) \leq \gamma_{b,2}(G_{m,n_0})$. We then exhaustively enumerate (computationally) all possible sub-broadcast induced by $r$ consecutive columns of $G_{m,n_0}$ of cost less than or equal to $s$. If it is possible to conclude that, for each possible sub-broadcast $g$, either $g$ cannot be a sub-broadcast of an optimal $2$-limited dominating broadcast on $G_{m,n_0}$ or $g$ forces $B(m, n_0) \leq \gamma_{b,2}(G_{m,n_0})$, then the desired bounds follows. See Example \ref{firstexample}.

\begin{eg}\label{firstexample}
Consider the following result.
\begin{prop}\cite[Theorem 3]{slobodinfinal}
For $n \geq 3$, 
\begin{equation*}
\begin{aligned}
\gamma_{b,2} \left( P_{5} \square C_{n} \right)   \leq n + \begin{cases}
0 & \textnormal{for } n \equiv 0 \pmod{ 2} \textnormal{ and} \\
1 & \textnormal{for } n \equiv 1 \pmod{ 2}.
\end{cases}
\end{aligned}
\end{equation*}
\end{prop}

Suppose we wish to obtain optimal values for $\gamma_{b,2} (P_5 \square C_n)$ when $n \equiv 0 \pmod{2} $ by proving that $n \leq \gamma_{b,2} (P_5 \square C_n)$. We have that $n \leq \gamma_{b,2} \left( P_{5} \square C_{n} \right) $ for $3 \leq n \leq 16$ by computation. Suppose the bound holds for all $n < n_0$ for some $n_0 > 16$. Let $f$ be a $2$-limited dominating broadcast of $P_{5} \square C_{n_0}$. Let $C$ be the subgraph of $ P_{5} \square C_{n_0} $ induced by the vertices appearing in a minimum cost (with respect to $f$) set of eight consecutive columns (here $r=8$). If the sub-broadcast induced by $V(C)$ has cost at least 8 (here $s = 7$), then $cost(f) \geq n_0$. It is therefore sufficient to consider, for each integer $ x \leq 7$, all possible sub-broadcasts of cost $x$ induced by $V(C)$. If it is possible to conclude that, for each such sub-broadcast $g$, either $g$ cannot be a sub-broadcast of an optimal $2$-limited dominating broadcast on $P_5 \square C_{n_0}$ or $g$ forces $n_0 \leq \gamma_{b,2}(P_5 \square C_{n_0})$, then the desired bounds follows.
\end{eg}

We conclude this section with definitions used throughout the rest of the paper.

\begin{defn}
\sloppy Given two $2$-limited broadcasts $f$ and $g$ on a graph $G$, for each $x \in V(G)$, define 
\begin{equation*}
\begin{aligned}
(f\oplus g)(x)= \textnormal{max} \left\{ f(x), g(x)\right\}
\end{aligned}
\end{equation*}
and
\begin{equation*}
\begin{aligned}
(f\ominus g)(x) &= \begin{cases}
0 & \textnormal{ if } g(x) > 0 \textnormal{ and} \\
f(x) & \textnormal{ otherwise.}
\end{cases}
\end{aligned}
\end{equation*}
\end{defn}

\begin{defn}
If $f$ is a broadcast on $G$, then we say $f$ \textit{dominates} $ y \in V(G)$ if there exists a vertex $x$ such that $f(x) \geq d(x,y)$. Further, we say that $f$ \textit{dominates} $X \subseteq V(G)$ if it dominates every vertex $x \in X$.
\end{defn}

\begin{defn}
Let $f$ be a broadcast on $G$. The \textit{broadcast range}  of $f$ is  the set of vertices which hear a broadcast under $f$.
\end{defn}

\section{Eliminating Possible Induced Sub-broadcasts of Fixed Cost}\label{sec:backtrack}

Fix $m$ and suppose we wish to establish the function $B(n)$ as a lower bound for the $2$-limited broadcast domination number of $G_{m,n}$. A positive fixed number $r =r(m) \geq 5$ of columns is chosen. In the inductive step, we consider $n_0 > r+10$ such that $B(n)$ is a lower bound for $\gamma_{b,2}(G_{m,n})$ for all $n < n_0$. Let $C$ be the subgraph of $G_{m,n_0}$ induced by the vertices of $r$ consecutive columns. We complete an exhaustive search of all possible sub-broadcasts $g$  induced by $V(C)$ and subject each such $g$ to a series of tests in the hope of excluding $g$ or concluding that $g$ forces $B(n_0) \leq \gamma_{b,2}(G_{m,n_0})$. Given $C$, four columns are added to both the left and right of $C$ in order to ensure that the subgraph considered is large enough to include all vertices dominated by any vertex that could potentially dominate some vertex in $C$. 

In summary, our algorithm takes as an input, $H_{m,k} = (P_m \textnormal{ or } C_m) \square P_k$, where $k = r+8$, with columns labelled $c_1, c_2, \dots, c_k$, where the vertices of $C$ are in columns $c_5, c_6, \dots,c_{k-4}$. Note that $(P_m \textnormal{ or } C_m) \square P_k$ is understood to mean $P_m \square P_{k}$ or $C_m \square P_{k}$ dependent upon $G_{m,n}$. Observe that $k \geq 13$ and $n_0 \geq k+3$. The assumptions defined previously are used in Sections \ref{middlesec} through \ref{lastsec} which describe and prove the correctness of the six schemes we use to eliminate sub-broadcasts. These schemes appear in the same order as in Algorithm \ref{alg:thefinalfinal} in Section \ref{sec:middle}.

\subsection{Domination Requirement}\label{middlesec}
Since we are looking for a lower bound for the cost of an optimal $2$-limited dominating broadcast $f$ on $G_{m,n_0}$, any induced sub-broadcast that forces vertices of $G_{m,n_0}$ to not be dominated, can be eliminated. 

\begin{Observation}\label{observation2}
If the sub-broadcast $g$ induced by $V(C)$ does not dominate the vertices of columns $c_7$, $c_8$, $\dots$, $c_{k-6}$, then $g$ cannot be a sub-broadcast of a dominating broadcast. 
\end{Observation}
See Figure \ref{fig:domcont}. The region containing $V(C)$ is depicted by the thick black rectangle. The black circles with a black inner fill indicate vertices broadcasting at a non-zero strength. The thick red dotted lines indicate the broadcast ranges of the broadcasting vertices at their centers. In this example, there is one vertex broadcasting at strength 2 in columns $c_4$ and $c_{k-3}$ and one vertex broadcasting at strength 1 in column $c_{k-3}$. Let $DoesNotDominate(H_{m,k}, g)$ return true if $g$ does not dominate the vertices of columns $c_7, c_8, \dots, c_{k-6}$ and false otherwise. 
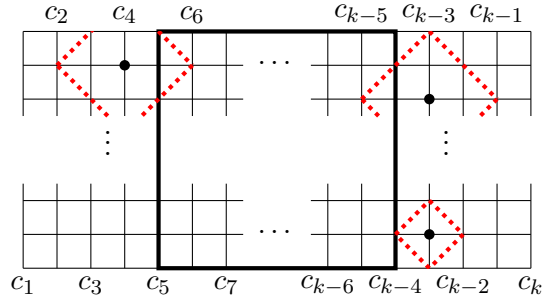
\begin{figure}[htbp]
\centering
\begin{tikzpicture}[scale = 0.42,
baseline={([yshift=-.5ex]current bounding box.center)},
vertex/.style = {circle, fill, inner sep=1.4pt, outer sep=0pt},
every edge quotes/.style = {auto=left, sloped, font=\scriptsize, inner sep=1pt}
]
\foreach \x in {2,4,6}{
\node at (\x-1,7.5) [label = center: $c_{\x}$] {};
}
\foreach \x in {1,3,5,7}{
\node at (\x-1,-0.5) [label = center: $c_{\x}$] {};
}
\node at (9,-0.5) [label = center: $c_{k-6}$] {};
\node at (11,-0.5) [label = center: $c_{k-4}$] {};
\node at (13,-0.5) [label = center: $c_{k-2}$] {};
\node at (15,-0.5) [label = center: $c_{k}$] {};
\node at (10,7.5) [label = center: $c_{k-5}$] {};
\node at (12,7.5) [label = center: $c_{k-3}$] {};
\node at (14,7.5) [label = center: $c_{k-1}$] {};
\node at (7.5,0.45) [label = : $\dots$] {};
\node at (7.5,5.45) [label = : $\dots$] {};
\node at (2.5,2.75) [label = : $\vdots$] {};
\node at (12.5,2.75) [label = : $\vdots$] {};
\draw[black, ultra thick] (4,0) rectangle (11,7);
\path
\foreach \x in {0,1,2,5,6,7}{
(0,\x) edge  (6.5,\x)
(8.5,\x) edge  (15,\x)
}
\foreach \y in {0,1,2,3,4,5,6,9,10,11,12,13,14,15}{
(\y,0) edge  (\y,2.5)
(\y,4.5) edge  (\y,7)
}
;
\node[vertex] at (12,1) {} ;
\node[vertex] at (3,6) {} ;
\node[vertex] at (12,5) {} ;
\clip (0,0) rectangle (16,7);
\path
(5, 6) edge [ultra thick, red, dotted] (3.5, 4.5)
(2.5, 4.5) edge [ultra thick, red, dotted] (1, 6)
(1, 6) edge [ultra thick, red, dotted] (3, 8)
(3, 8) edge [ultra thick, red, dotted] (5, 6)
(14, 5) edge [ultra thick, red, dotted] (13.5, 4.5)
(10.5, 4.5) edge [ultra thick, red, dotted] (10, 5)
(10, 5) edge [ultra thick, red, dotted] (12, 7)
(12, 7) edge [ultra thick, red, dotted] (14, 5)
(11, 1) edge [ultra thick, red, dotted] (12, 2)
(12, 2) edge [ultra thick, red, dotted] (13, 1)
(13, 1) edge [ultra thick, red, dotted] (12, 0)
(12, 0) edge [ultra thick, red, dotted] (11, 1)
;
\end{tikzpicture}
\caption{The graph $H_{m,k}$ with columns labelled $c_1, c_2, \dots, c_k$, $C$ indicated by the thick black rectangle, and possible broadcast vertices exterior to $C$ which can only dominate vertices in columns $c_5, c_6, c_{k-5},$ and $c_{k-4}$.}
\label{fig:domcont}
\end{figure}

\subsection{Forbidden Broadcasts}\label{sec:forbid}
To improve the speed of our computations, we have identified four simple forbidden broadcast structures. 
\begin{figure}[htbp]
\centering
\begin{tikzpicture}[scale = 0.42,
baseline={([yshift=-.5ex]current bounding box.center)},
vertex/.style = {circle, fill, inner sep=1.4pt, outer sep=0pt},
every edge quotes/.style = {auto=left, sloped, font=\scriptsize, inner sep=1pt}
]
\path
\foreach \x in {-2,-1,0,1,2}{
(4.5,\x) edge  (9.5,\x)
}
\foreach \x in {5,6,7,8,9}{
(\x,-2.5) edge  (\x,2.5)
};
\path
(9, 0) edge [ultra thick, dotted, red] (7, -2)
(7, -2) edge [ ultra thick, dotted,red] (5, 0)
(5, 0) edge [ultra thick, dotted,red] (7, 2)
(7, 2) edge [ultra thick, dotted,red] (9, 0)
(7, 0) edge [ultra thick, dotted,red] (8, 1)
(8, 1) edge [ultra thick, dotted,red] (9, 0)
(9, 0) edge [ultra thick, dotted,red] (8, -1)
(8, -1) edge [ultra thick, dotted,red] (7, 0)
;
\node[vertex] at (7,0) {} ;
\node[vertex] at (8,0) {} ;
\node at (7,-4) [label=$a)$]{};
\clip (4.5,-4) rectangle (9.5,2.5);
\end{tikzpicture}
\hspace{0.5cm}
\begin{tikzpicture}[scale = 0.42,
baseline={([yshift=-.5ex]current bounding box.center)},
vertex/.style = {circle, fill, inner sep=1.4pt, outer sep=0pt},
every edge quotes/.style = {auto=left, sloped, font=\scriptsize, inner sep=1pt}
]
\path
\foreach \x in {-2,-1,0,1,2}{
(4.5,\x) edge  (9.5,\x)
}
\foreach \x in {5,6,7,8,9}{
(\x,-2.5) edge  (\x,2.5)
};
\path
(5, 0) edge [ultra thick, dotted,red] (6, 1)
(6, 1) edge [ultra thick, dotted,red] (7, 0)
(7, 0) edge [ultra thick, dotted,red] (6, -1)
(6, -1) edge [ultra thick, dotted,red] (5, 0)
(7, 0) edge [ultra thick, dotted,red] (8, 1)
(8, 1) edge [ultra thick, dotted,red] (9, 0)
(9, 0) edge [ultra thick, dotted,red] (8, -1)
(8, -1) edge [ultra thick, dotted,red] (7, 0)
;
\node[vertex] at (6,0) {} ;
\node[vertex] at (8,0) {} ;
\node[vertex] at (7,0) {} ;
\node at (6.5,-2.15) [label=$w$]{};
\draw[->, thick] (6.6,-1.2) -- (6.9,-0.2);
\node at (7,-4) [label=$b)$]{};
\clip (4.5,-4) rectangle (9.5,2.5);
\end{tikzpicture}
\hspace{0.5cm}
\begin{tikzpicture}[scale = 0.42,
baseline={([yshift=-.5ex]current bounding box.center)},
vertex/.style = {circle, fill, inner sep=1.4pt, outer sep=0pt},
every edge quotes/.style = {auto=left, sloped, font=\scriptsize, inner sep=1pt}
]
\path
\foreach \x in {-2,-1,0,1,2}{
(4.5,\x) edge  (9.5,\x)
}
\foreach \x in {5,6,7,8,9}{
(\x,-2.5) edge  (\x,2.5)
};
\path
(6, 1) edge [ultra thick, dotted,red] (7, 2)
(7, 2) edge [ultra thick, dotted,red] (8, 1)
(8, 1) edge [ultra thick, dotted,red] (7, 0)
(7, 0) edge [ultra thick, dotted,red] (6, 1)
(7, 0) edge [ultra thick, dotted,red] (8, 1)
(8, 1) edge [ultra thick, dotted,red] (9, 0)
(9, 0) edge [ultra thick, dotted,red] (8, -1)
(8, -1) edge [ultra thick, dotted,red] (7, 0)
;
\node[vertex] at (7,1) {} ;
\node[vertex] at (8,0) {} ;
\node[vertex] at (7,0) {} ;
\node at (6.5,-2.15) [label=$w$]{};
\draw[->, thick] (6.5,-1.2) -- (6.9,-0.2);
\node at (7,-4) [label=$c)$]{};
\clip (4.5,-4) rectangle (9.5,2.5);
\end{tikzpicture}
\hspace{0.5cm}
\begin{tikzpicture}[scale = 0.42,
baseline={([yshift=-.5ex]current bounding box.center)},
vertex/.style = {circle, fill, inner sep=1.4pt, outer sep=0pt},
every edge quotes/.style = {auto=left, sloped, font=\scriptsize, inner sep=1pt}
]
\path
\foreach \x in {-2,-1,0,1,2}{
(4.5,\x) edge  (9.5,\x)
}
\foreach \x in {5,6,7,8,9}{
(\x,-2.5) edge  (\x,2.5)
};
\path
(6, 0) edge [ultra thick, dotted,red] (7, 1)
(7, 1) edge [ultra thick, dotted,red] (8, 0)
(8, 0) edge [ultra thick, dotted,red] (7, -1)
(7, -1) edge [ultra thick, dotted,red] (6, 0)

(7, 0) edge [ultra thick, dotted,red] (8, 1)
(8, 1) edge [ultra thick, dotted,red] (9, 0)
(9, 0) edge [ultra thick, dotted,red] (8, -1)
(8, -1) edge [ultra thick, dotted,red] (7, 0)
;
\node[vertex] at (8,0) {} ;
\node[vertex] at (7,0) {} ;
\node at (6.5,-2.15) [label=$w$]{};
\draw[->, thick] (6.5,-1.2) -- (6.9,-0.2);
\node at (7,-4) [label=$d)$]{};
\clip (4.5,-4) rectangle (9.5,2.5);
\end{tikzpicture}
\caption{Forbidden broadcasts.}
\label{fig:wendy}
\end{figure}
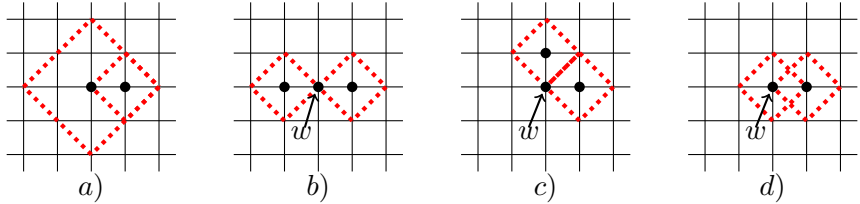
Without loss of generality, a sub-broadcast $g$ cannot contain any of the forbidden broadcasts shown in Figure \ref{fig:wendy} because the broadcast in $a)$ cannot be found in an optimal $2$-limited broadcast, and  the broadcasts in $b)$, $c)$, and $d)$ can be replaced by a broadcast of strength 2 from $w$ while preserving the cost of $g$ and extending the range of $g$. Let $ForbiddenBroadcast(H_{m,k}, g)$ return true if $g$ exhibits $a)$, $b)$, $c)$, or $d)$ and false otherwise.

\subsection{Optimality Requirement}\label{sec:backtrack:optimality}
As we are attempting to prove a lower bound for the cost of an optimal $2$-limited dominating broadcast on $G_{m,n_0}$, any possible sub-broadcast that is not optimal can be eliminated.

\begin{Observation}\label{observation3}
If the broadcast range $R$ of a possible sub-broadcast $g$ induced by $V(C)$ can be dominated by a broadcast $h$ on $H_{m,k}$ of cost strictly less than $cost(g)$, then $g$ cannot be a sub-broadcast of an optimal dominating broadcast.
\end{Observation}

Let $HasBroadcast(H_{m,k},R,x)$ return true if $R$ can be dominated with cost less than or equal to $x$ on $H_{m,k}$ and false otherwise.

\subsection{Proof by Induction}\label{sec:backtrack:actualmethods}
Recall that, after checking the appropriate base cases, we assume $B(n)$ is a lower bound of $\gamma_{b,2}(G_{m,n})$ for all $n < n_0$ where $n_0 \geq k+3$. Additionally, the values $r = k-8$ and $s$ are chosen such that, if the minimum cost of a sub-broadcast induced by $r$ consecutive columns of $G_{m,n_0}$ is strictly greater than $s$, then $B(n_0) \leq \gamma_{b,2}(G_{m,n_0})$. As such, we may be able to conclude (via the inductive assumption) that possible sub-broadcasts $g$ of cost less than or equal to $s$ induced by $r$ consecutive columns of $G_{m,n_0}$ imply the bound we hope to prove. This can be done by deleting $i \leq k$ columns and ``patching'' the graph back together to obtain a $2$-limited dominating broadcast  on $G_{m,n_0-i}$, the cost of which we assumed to be greater than or equal to $B(n_0-i)$.  In general, given some possible induced sub-broadcast $g$ whose broadcast range $R$ is contained within $k$ consecutive columns, we delete a particular selection of $i$ columns, for each $i$ from 1 to $k$. To this end, define 
\begin{equation*}
m_i = 
\begin{cases}
    \max_{n \geq n_0} \left\{ B(n) - B(n-i) \right\} & \textnormal{ should it exist and} \\
\infty & \textnormal{otherwise.} 
\end{cases}
\end{equation*}
Algorithm \ref{alg:indcheck} and Lemma \ref{thm:alg} formalize our approach. See Example \ref{eg_inductin_confusing} for an illustration of this test on $P_5 \square C_n$.

\begin{algorithm}[htbp]
\SetKwInOut{Input}{Input}
\SetKwInOut{Output}{Output}
\underline{function InductiveArgument} $(H_{m,k}, R, x, m_1, m_2,\dots, m_k)$\;
\Input{\sloppy A graph $H_{m,k} = ( P_m \textnormal{ or } C_m) \square P_k $ with columns labelled from left to right by $c_1, c_2, \dots, c_k$, and rows from 1 to $m$, a set of vertices $R \subseteq V(H_{m,k})$ labelled according to their row number and column label, the cost $x$  used to dominate $R$ by some $2$-limited broadcast $g$ whose broadcast range lies entirely within $H_{m,k}$, and $m_i$ (for each $i=1$ to $k$) as defined in Section \ref{sec:backtrack:actualmethods}.}
\Output{\sloppy Value of the truth statement: ``sub-broadcast implies bound.''}
Create a sorted list $L$ of the columns that contain at least one vertex of $R$ so that they are first ordered from maximum to minimum according to the number of vertices that are in $R$. Resolve ties by sorting so that $c_i$ comes before $c_j$ if $i < j$\;
\For{$i $ \textnormal{\textbf{from}}  $1$ \textnormal{\textbf{to}} $length(L)$}{
Let $S$ be the set of columns contained in the first $i$ entries of list $L$\;
Let $H_{m,k-|S|} = (P_m \textnormal{ or } C_m) \square P_{k-|S|}$\;
Set $R' = \left\{ v \in R: v \in H_{m,k-|S|}  \right\} $\;
\textbf{if} $HasBroadcast(H_{m,k-|S|},R',x-m_i)$ \textbf{then return} True\;
}
\Return False\;
\caption{Routine to determine if assumed sub-broadcast implies bound.}
\label{alg:indcheck}
\end{algorithm}

\begin{lem}\label{thm:alg}
    \sloppy Let $k \geq 13$, $B(n)$ be a lower bound of $\gamma_{b,2}(G_{m,n})$ for all $n < n_0$ for some $n_0\geq k+3$, and let $g$ be a broadcast of cost $x$ whose broadcast range $R$ is contained in some $k$-column induced subgraph $H_{m,k}$ of $G_{m,n_0}$. If $g$ is a sub-broadcast of an optimal broadcast on $G_{m,n_0}$ and $InductiveArgument(H_{m,k},$ $R,$ $x,$ $m_1,$ $m_2,$ $\dots,$ $m_k)$ is true, then $B(n_0) \leq \gamma_{b,2}(G_{m,n_0})$. Here $m_i$ (for each $i=1$ to $k$) is defined as in Section \ref{sec:backtrack:actualmethods}. 
\end{lem}

\begin{proof}
\sloppy Assume the conditions of Lemma \ref{thm:alg}. Let $f$ be an optimal $2$-limited dominating broadcast of $G_{m,n_0}$ which contains $g$ as an induced sub-broadcast. Let $f' = f \ominus g$ and let  $S$, $H_{m,k-|S|}$, and $R'$ be defined as in Algorithm \ref{alg:indcheck}. Let $G_{m,n_0-|S|}$ be constructed by removing the set of columns $S$ from $G_{m,n_0}$ which resulted in $InductiveArgument$ returning true on line 7 and adding edges in the natural way such that the resulting graph is isomorphic to $ \left( P_m \textnormal{ or }C_m  \right)  \square C_{n_0- |S|} $.

\begin{Observation}\label{observation1}
As $n_0 \geq k+3$ and $\lvert S \rvert \leq k$, we have that  $n_0- \lvert S \rvert \geq 3$. Thus, $G_{m,n_0-\lvert S \rvert }$ is well-defined.
\end{Observation}

Let $f''$ be the broadcast formed by restricting $f'$ to $G_{m,n_0-|S|}$. That is, let $f''(v) = f'(v)$ for all $v \in V(G_{m,n_0-|S|})$. For each vertex $v \in V(G_{m,n_0}) \setminus V(G_{m,n_0-|S|})$ (i.e. all vertices $v$ in the set of columns $S$) broadcasting with non-zero strength under $f'$, pick a vertex $u$ in the same row as $v$ and in an undeleted column nearest to $v$ and let $f''(u) = \textnormal{ max}  \left\{ f'(u),f'(v) \right\} $. Note that $cost(f'') \leq cost(f')$.

\begin{Observation}\label{observation}
The broadcast $f''$ dominates $V(G_{m,n_0-|S|})$ with the possible exception of the vertices of $R'$.
\end{Observation}
\begin{proof}
Fix $v \in V(G_{m,n_0-|S|})\setminus R'$. As $v \not \in R'$, $v$ is not dominated by $g$ on $G_{m,n_0}$. There are two cases:

Case \Romannum{1}, $v$ hears a broadcast from a vertex $u$ under $f'$ and $u$ is not in the set of columns $S$. As $f''(u) \geq f'(u)$, since $v$ hears a broadcast from $u$ under $f'$ on $G_{m,n_0}$, $v$ hears a broadcast from $u$ under $f''$ on $G_{m,n_0-|S|}$. 

Case \Romannum{2}, $v$ hears a broadcast from a vertex $u$ under $f'$ and $u$ is in the set of columns $S$. As $u \not \in V(G_{m,n_0-|S|})$, there exists a vertex $u' \in V(G_{m,n_0-|S|})$ in the same row as $u$ and in a nearest column undeleted to $u$ such that $f''(u') \geq f'(u)$. Note that, by Observation \ref{observation1}, such a vertex exists. It suffices to check that $v$ hears a broadcast from $u'$ under $f''$ on $G_{m,n_0-|S|}$. Since $f'(u) \leq 2$ and $G_{m,n_0}= (P_m \textnormal{ or } C_m) \square C_{n_0}$, there are two subcases:

Subcase \Romannum{2}.a),  $u'$ is on the same column as $v$ or $u'$ is on the column between $v$ and $u$ on $G_{m,n_0}$. Since $u'$ is in the same row as $u$, $d(u',v) \leq d(u,v)$ in $G_{m,n_0}$. As $f''(u') \geq f'(u)$, since $v$ hears a broadcast from $u$ under $f'$ on $G_{m,n_0}$, $v$ hears a broadcast from $u'$ under $f''$ on $G_{m,n_0-|S|}$.

Subcase \Romannum{2}.b), $u$ is on a column between $v$ and $u'$ on $G_{m,n_0}$. As $u'$ is on a nearest column to $u$ that is undeleted, the column of $u'$ on $G_{m,n_0-|S|}$ is at distance at most $d_{G_{m,n_0}}(u,v)$ (the distance between $u$ and $v$ in $G_{m,n_0}$) from the column containing $v$. As $u'$ is in the same row as $u$ and $f''(u') \geq f'(u)$, since $v$ hears a broadcast from $u$ under $f'$ on $G_{m,n_0}$, $v$ hears a broadcast from $u'$ under $f''$ on $G_{m,n_0-|S|}$.
\end{proof}

By Observation \ref{observation}, to dominate $G_{m,n_0-|S|}$, it suffices to find a $2$-limited broadcast which dominates $R'$. As the function call on line 7 returns true, $R'$ can be dominated by a $2$-limited broadcast $h$ with cost less than or equal to $ x-m_{|S|}$ where $x = cost(g)$. Note that, for this to be true, $m_{|S|} \neq \infty$. Let $f'''= f'' \oplus h$. The broadcast $f'''$ is a $2$-limited dominating broadcast on $G_{m,n_0-|S|}$. By the inductive assumption, 
\begin{equation}\label{eqfirst}
B(n_0 - |S|) \leq \gamma_{b,2}(G_{m,n_0 - |S|}) \leq cost(f''').
\end{equation}
Observe that
\begin{equation}\label{eqsecond}
\begin{aligned}
    cost(f''') = cost(f''\oplus h) = cost(f'') + cost(h) &\leq cost(f') + cost(h)\\
                                   &\leq cost(f') + x - m_{|S|}\\
                                   &= cost(f\ominus g) + x-m_{|S|}\\
                                   &= cost(f) - cost(g) + x - m_{|S|}\\
                                   &= cost(f) - m_{|S|}. 
\end{aligned}
\end{equation}
As $m_{|S|} \geq B(n_0) - B(n_0-|S|)$, when combined with Equations \ref{eqfirst} and \ref{eqsecond}, we have that
\begin{equation*}
\begin{aligned}
    B(n_0) \leq B(n_0-|S|) + m_{|S|} \leq  cost(f''') + m_{|S|}\leq cost(f) = \gamma_{b,2}(G_{m,n_0})
\end{aligned}
\end{equation*}
as desired.

\end{proof}

\begin{eg}\label{eg_inductin_confusing}
Suppose we wish to establish $n \leq \gamma_{b,2} (P_5 \square C_n)$; doing so will yield periodically optimal values for $\gamma_{b,2} (P_5 \square C_n)$. By computation, we have that  $n \leq \gamma_{b,2} \left( P_{5} \square C_{n} \right) $ for $3 \leq n \leq 16$. Suppose the bound holds for all $n < n_0$ where $n_0 > 16$. Let $f$ be an optimal $2$-limited dominating broadcast of $P_5 \square C_{n_0}$ and let $C$ be the subgraph of $ P_{5} \square C_{n_0}$ induced by the vertices appearing in a minimum cost set of eight consecutive columns with respect to $f$. Suppose $V(C)$ induces the sub-broadcast $g$ shown in Figure \ref{fig:backtrack:inductionbasecontraintro}.
\begin{figure}[htbp]
\centering
\begin{tikzpicture}[scale = 0.42,
baseline={([yshift=-.5ex]current bounding box.center)},
vertex/.style = {circle, fill, inner sep=1.4pt, outer sep=0pt},
every edge quotes/.style = {auto=left, sloped, font=\scriptsize, inner sep=1pt}
]
\draw[black, ultra thick] (4,0) rectangle (11,4);
\node[vertex] at (7,0) {} ;
\node[vertex] at (9,2) {} ;
\node[vertex] at (6,3) {} ;
\node[vertex] at (8,4) {} ;
\node[vertex] at (4,1) {} ;
\node[vertex] at (11,1) {} ;
\path
\foreach \x in {0,1,2,3,4}{
(1.5,\x) edge  (13.5,\x)
}
\foreach \y in {2,3,4,5,6,7,8,9,10,11,12,13}{
(\y,0) edge  (\y,4)
}
;
\clip (1.5,0) rectangle (13.5,4);
\path
(9, 0) edge [ultra thick, dotted, red] (7, -2)
(7, -2) edge [ ultra thick, dotted,red] (5, 0)
(5, 0) edge [ultra thick, dotted,red] (7, 2)
(7, 2) edge [ultra thick, dotted,red] (9, 0)
(10, 2) edge [ ultra thick, dotted, red] (9, 1)
(9, 1) edge [ ultra thick, dotted,red] (8, 2)
(8, 2) edge [ultra thick, dotted,red] (9, 3)
(9, 3) edge [ultra thick, dotted,red] (10, 2)
(7, 3) edge [ultra thick, dotted,red] (6, 2)
(6, 2) edge [ultra thick, dotted,red] (5, 3)
(5, 3) edge [ultra thick, dotted,red] (6, 4)
(6, 4) edge [ultra thick, dotted,red] (7, 3)
(9, 4) edge [ultra thick, dotted,red] (8, 3)
(8, 3) edge [ultra thick, dotted,red] (7, 4)
(7, 4) edge [ultra thick, dotted,red] (8, 5)
(8, 5) edge [ultra thick, dotted,red] (9, 4)
(4, 0) edge [ultra thick, dotted,red] (5, 1)
(5, 1) edge [ultra thick, dotted,red] (4, 2)
(4, 2) edge [ultra thick, dotted,red] (3, 1)
(3, 1) edge [ultra thick, dotted,red] (4, 0)
(11, 0) edge [ultra thick, dotted,red] (12, 1)
(12, 1) edge [ultra thick, dotted,red] (11, 2)
(11, 2) edge [ultra thick, dotted,red] (10, 1)
(10, 1) edge [ultra thick, dotted,red] (11, 0)
;
\end{tikzpicture}
\caption{Assumed sub-broadcast $g$ induced by $V(C)$ of cost 7.}
\label{fig:backtrack:inductionbasecontraintro}
\end{figure}
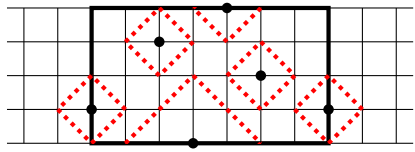

Let $f'=f\ominus g$ and let $R$ be the range of $g$. The broadcast $f'$ dominates $V(P_5 \square C_{n_0})$ with the possible exception of the vertices of $R$. Suppose we delete the four columns indicated in Figure \ref{fig:backtrack:inductionbasecontraintrosec} (Left) from the grid and add edges in the natural way such that the resulting graph $G_{5,n_0-4}= P_{5} \square C_{n_0-4} $. 
\begin{figure}[htbp]
\centering
\begin{tikzpicture}[scale = 0.42,
baseline={([yshift=-.5ex]current bounding box.center)},
vertex/.style = {circle, fill, inner sep=1.4pt, outer sep=0pt},
every edge quotes/.style = {auto=left, sloped, font=\scriptsize, inner sep=1pt}
]
\draw [decorate,decoration={brace,amplitude=5pt,raise=2pt},xshift=0pt]
(6,4) -- (9,4) node [black,midway,yshift=0.5cm] {columns to be deleted};
\draw[black, ultra thick] (4,0) rectangle (11,4);
\draw[ForestGreen, ultra thick] (5,0) circle  (10pt);
\draw[ForestGreen, ultra thick] (4,0) circle  (10pt);
\draw[ForestGreen, ultra thick] (4,1) circle  (10pt);
\draw[ForestGreen, ultra thick] (4,2) circle  (10pt);
\draw[ForestGreen, ultra thick] (3,1) circle  (10pt);
\draw[ForestGreen, ultra thick] (5,1) circle  (10pt);
\draw[ForestGreen, ultra thick] (5,3) circle  (10pt);
\draw[ForestGreen, ultra thick] (10,2) circle  (10pt);
\draw[ForestGreen, ultra thick] (11,2) circle  (10pt);
\draw[ForestGreen, ultra thick] (11,1) circle  (10pt);
\draw[ForestGreen, ultra thick] (11,0) circle  (10pt);
\draw[ForestGreen, ultra thick] (10,1) circle  (10pt);
\draw[ForestGreen, ultra thick] (12,1) circle  (10pt);

\begin{pgfonlayer}{background}
\clip (1.5,0) rectangle (13.5,4);
\node[2] at (7,0) {} ;
\node[1] at (9,2) {} ;
\node[1] at (6,3) {} ;
\node[1] at (8,4) {} ;
\node[1] at (11,1) {} ;
\node[1] at (4,1) {} ;
\end{pgfonlayer}

\draw[fill = white, opacity = 0.1] (6,0) rectangle (9,4);
\draw[thick, pattern=north east lines,pattern color=black] (6,0) rectangle (9,4);
\path
\foreach \x in {0,1,2,3,4}{
(1.5,\x) edge  (13.5,\x)
}
\foreach \y in {2,3,4,5,6,7,8,9,10,11,12,13}{
(\y,0) edge  (\y,4)
}
;
\clip (1.5,0) rectangle (13.5,4);

\end{tikzpicture}
\hspace{0.25cm}
$\rightarrow$
\begin{tikzpicture}[scale = 0.42,
baseline={([yshift=-.5ex]current bounding box.center)},
vertex/.style = {circle, fill, inner sep=1.4pt, outer sep=0pt},
every edge quotes/.style = {auto=left, sloped, font=\scriptsize, inner sep=1pt}
]
\draw [decorate,decoration={brace,amplitude=5pt,raise=2pt},xshift=0pt]
(4,4) -- (7,4) node [black,midway,yshift=0.5cm] {resulting region after deletion};
\draw[ForestGreen, ultra thick] (5,0) circle  (10pt);
\draw[ForestGreen, ultra thick] (4,0) circle  (10pt);
\draw[ForestGreen, ultra thick] (4,1) circle  (10pt);
\draw[ForestGreen, ultra thick] (4,2) circle  (10pt);
\draw[ForestGreen, ultra thick] (3,1) circle  (10pt);
\draw[ForestGreen, ultra thick] (5,1) circle  (10pt);
\draw[ForestGreen, ultra thick] (5,3) circle  (10pt);
\draw[ForestGreen, ultra thick] (6,2) circle  (10pt);
\draw[ForestGreen, ultra thick] (7,2) circle  (10pt);
\draw[ForestGreen, ultra thick] (7,1) circle  (10pt);
\draw[ForestGreen, ultra thick] (7,0) circle  (10pt);
\draw[ForestGreen, ultra thick] (6,1) circle  (10pt);
\draw[ForestGreen, ultra thick] (8,1) circle  (10pt);
\path
\foreach \x in {0,1,2,3,4}{
(1.5,\x) edge  (9.5,\x)
}
\foreach \y in {2,3,4,5,6,7,8,9}{
(\y,0) edge  (\y,4)
}
;
\end{tikzpicture}
\hspace{0.25cm}
\begin{tikzpicture}[scale = 0.42,
baseline={([yshift=-.5ex]current bounding box.center)},
vertex/.style = {circle, fill, inner sep=1.4pt, outer sep=0pt},
every edge quotes/.style = {auto=left, sloped, font=\scriptsize, inner sep=1pt}
]
\phantom{\draw [decorate,decoration={brace,amplitude=5pt,raise=2pt},xshift=0pt]
    (4,4) -- (7,4) node [black,midway,yshift=0.5cm] {};}

\path
\foreach \x in {0,1,2,3,4}{
(1.5,\x) edge  (9.5,\x)
}
\foreach \y in {2,3,4,5,6,7,8,9}{
(\y,0) edge  (\y,4)
};
\clip (1.5,0) rectangle (9.5,4);
\path
(3,1) edge [ultra thick, dotted, red] (5, 3)
(5, 3) edge [ ultra thick, dotted,red] (7, 1)
(7, 1) edge [ultra thick, dotted,red] (5, -1)
(5, -1) edge [ultra thick, dotted,red] (3, 1)
(6,1) edge [ultra thick, dotted, red] (7, 2)
(7, 2) edge [ ultra thick, dotted,red] (8, 1)
(8, 1) edge [ultra thick, dotted,red] (7, 0)
(7, 0) edge [ultra thick, dotted,red] (6, 1)
;
\node[vertex] at (5,1) {} ;
\node[vertex] at (7,1) {} ;
\end{tikzpicture}
\caption{(Left \& Middle) Procedure which reduces $ G_{m,n_0} $ to $ G_{m,n_0-4} $. (Right) Broadcast of cost 3 which dominates $R'$.}
\label{fig:backtrack:inductionbasecontraintrosec}
\end{figure}
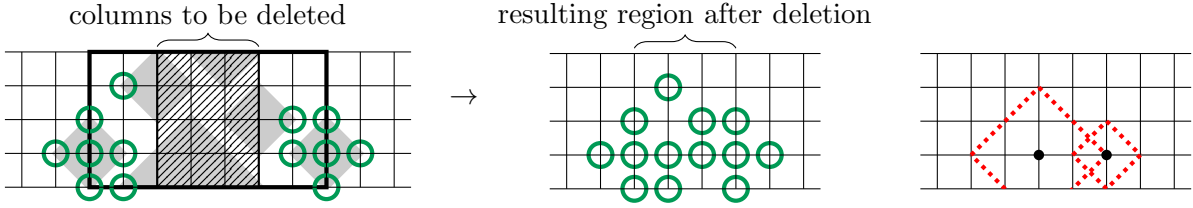
Let $f''$ be the broadcast formed by restricting $f'$ to $G_{m,n_0-4}$. That is, let $f''(v) = f'(v)$ for all $v \in V(G_{m,n_0-4})$. Let $R' = \left\{ v \in V(G_{m,n_0-4}) : v \in R  \right\} $. The vertices of $R'$ are indicated by the green circles in Figures \ref{fig:backtrack:inductionbasecontraintrosec} (Middle). The broadcast $f''$ dominates $V(G_{m,n_0-4})$ with the possible exception of the vertices of $R'$. However, $R'$ can be dominated by the broadcast $h$ of cost 3 as shown in Figure \ref{fig:backtrack:inductionbasecontraintrosec} (Right). Let $f''' = f'' \oplus h$. The broadcast $f'''$ is a $2$-limited dominating broadcast on $G_{5,n_0-4} = P_5 \square C_{n_0-4}$. By assumption, $n_0-4 \leq \gamma_{b,2}(P_5 \square C_{n_0-4}) \leq cost(f''')$. As
\begin{equation*}
\begin{aligned}
cost(f''') \leq cost(f'') +3 = cost(f\ominus g) + 3  \leq cost(f) - 4,
\end{aligned}
\end{equation*}
we conclude that $n_0 \leq cost(f) = \gamma_{b,2}(P_5 \square C_{n_0})$ as desired.
\end{eg}

\subsection{Necessary Broadcasts}\label{subsec:necbroad}
Given some possible sub-broadcast $g$, the broadcast structure of $g$ may imply the existence of a larger sub-broadcast $g'$. If this broadcast $g'$ does not pass the optimality requirement (see Section \ref{sec:backtrack:optimality}), then $g$ can be eliminated. If this broadcast $g'$ allows for the induction argument (see Section \ref{sec:backtrack:actualmethods}), then $g$ implies the bound we hope to prove.

\begin{Observation} 
    As $g$ is induced by $V(C)$, for each vertex in column $c_6$ or $c_{k-5}$ not dominated by $g$, any dominating broadcast of $G_{m,n_0}$ must have a vertex broadcasting at strength 2, in the same row and in column $c_4$ or $c_{k-3}$.  
\end{Observation}

\sloppy See Figure \ref{fig:backtrack:notuniquefailedmaxregion} (Left). Let $g'$ be the broadcast constructed from $g$ by adding these necessary broadcasts. Let $R'$ be the broadcast range of $g'$. Let $NecessaryBroadcast(H_{m,k},$ $g,$ $m_1,$ $m_2,$ $\dots,$ $m_k)$ return true if either $HasBroadcast(H_{m,k},$ $R',$ $cost(g')-1)$ or $InductiveArgument (H_{m,k},$ $R',$ $cost(g'),$ $m_1,$ $m_2,$ $\dots,$ $m_k)$ returns true and false otherwise.

\begin{figure}[htbp]
\centering
\begin{tikzpicture}[scale = 0.42,
baseline={([yshift=-.5ex]current bounding box.center)},
vertex/.style = {circle, fill, inner sep=1.4pt, outer sep=0pt},
every edge quotes/.style = {auto=left, sloped, font=\scriptsize, inner sep=1pt}
]
\foreach \x in {2,4,6}{
\node at (\x-1,7.5) [label = center: $c_{\x}$] {};
}
\foreach \x in {1,3,5,7}{
\node at (\x-1,-0.5) [label = center: $c_{\x}$] {};
}
\node at (9,-0.5) [label = center: $c_{k-6}$] {};
\node at (11,-0.5) [label = center: $c_{k-4}$] {};
\node at (13,-0.5) [label = center: $c_{k-2}$] {};
\node at (15,-0.5) [label = center: $c_{k}$] {};
\node at (10,7.5) [label = center: $c_{k-5}$] {};
\node at (12,7.5) [label = center: $c_{k-3}$] {};
\node at (14,7.5) [label = center: $c_{k-1}$] {};

\node at (7.5,0.45) [label = : $\dots$] {};
\node at (7.5,5.45) [label = : $\dots$] {};
\node at (2.5,2.75) [label = : $\vdots$] {};
\node at (12.5,2.75) [label = : $\vdots$] {};
\draw[BurntOrange, ultra thick] (5,0) circle  (10pt);
\draw[BurntOrange, ultra thick] (5,6) circle  (10pt);
\draw[BurntOrange, ultra thick] (10,5) circle  (10pt);
\draw[black, ultra thick] (4,0) rectangle (11,7);
\path
\foreach \x in {0,1,2,5,6,7}{
(0,\x) edge  (6.5,\x)
(8.5,\x) edge  (15,\x)
}
\foreach \y in {0,1,2,3,4,5,6,9,10,11,12,13,14,15}{
(\y,0) edge  (\y,2.5)
(\y,4.5) edge  (\y,7)
}
;
\node[vertex] at (3,0) {} ;
\node[vertex] at (3,6) {} ;
\node[vertex] at (12,5) {} ;

\clip (0,0) rectangle (16,7);
\path
(5, 0) edge [ultra thick, red, dotted] (3, -2)
(3, -2) edge [ultra thick, red, dotted] (1, 0)
(1, 0) edge [ultra thick, red, dotted] (3, 2)
(3, 2) edge [ultra thick, red, dotted] (5, 0)
(5, 6) edge [ultra thick, red, dotted] (3.5, 4.5)
(2.5, 4.5) edge [ultra thick, red, dotted] (1, 6)
(1, 6) edge [ultra thick, red, dotted] (3, 8)
(3, 8) edge [ultra thick, red, dotted] (5, 6)
(14, 5) edge [ultra thick, red, dotted] (13.5, 4.5)
(10.5, 4.5) edge [ultra thick, red, dotted] (10, 5)
(10, 5) edge [ultra thick, red, dotted] (12, 7)
(12, 7) edge [ultra thick, red, dotted] (14, 5)
;
\end{tikzpicture}
\hspace{0.5cm}
\begin{tikzpicture}[scale = 0.42,
baseline={([yshift=-.5ex]current bounding box.center)},
vertex/.style = {circle, fill, inner sep=1.4pt, outer sep=0pt},
every edge quotes/.style = {auto=left, sloped, font=\scriptsize, inner sep=1pt}
]
\foreach \x in {2,4,6}{
\node at (\x-1,7.5) [label = center: $c_{\x}$] {};
}
\foreach \x in {1,3,5,7}{
\node at (\x-1,-0.5) [label = center: $c_{\x}$] {};
}
\node at (9,-0.5) [label = center: $c_{k-6}$] {};
\node at (11,-0.5) [label = center: $c_{k-4}$] {};
\node at (13,-0.5) [label = center: $c_{k-2}$] {};
\node at (15,-0.5) [label = center: $c_{k}$] {};
\node at (10,7.5) [label = center: $c_{k-5}$] {};
\node at (12,7.5) [label = center: $c_{k-3}$] {};
\node at (14,7.5) [label = center: $c_{k-1}$] {};

\node at (7.5,0.45) [label = : $\dots$] {};
\node at (7.5,5.45) [label = : $\dots$] {};
\node at (2.5,2.75) [label = : $\vdots$] {};
\node at (12.5,2.75) [label = : $\vdots$] {};
\draw[BurntOrange, ultra thick] (5,0) circle  (10pt);
\draw[BurntOrange, ultra thick] (5,6) circle  (10pt);
\draw[BurntOrange, ultra thick] (10,5) circle  (10pt);
\draw[ForestGreen, ultra thick] (11,1) circle  (10pt);
\draw[black, ultra thick] (4,0) rectangle (11,7);
\path
\foreach \x in {0,1,2,5,6,7}{
(0,\x) edge  (6.5,\x)
(8.5,\x) edge  (15,\x)
}
\foreach \y in {0,1,2,3,4,5,6,9,10,11,12,13,14,15}{
(\y,0) edge  (\y,2.5)
(\y,4.5) edge  (\y,7)
}
;

\node[vertex] at (3,0) {} ;
\node[vertex] at (3,6) {} ;
\node[vertex] at (12,5) {} ;
\node[vertex] at (12,1) {} ;

\clip (0,0) rectangle (16,7);
\path
(5, 0) edge [ultra thick, red, dotted] (3, -2)
(3, -2) edge [ultra thick, red, dotted] (1, 0)
(1, 0) edge [ultra thick, red, dotted] (3, 2)
(3, 2) edge [ultra thick, red, dotted] (5, 0)
(5, 6) edge [ultra thick, red, dotted] (3.5, 4.5)
(2.5, 4.5) edge [ultra thick, red, dotted] (1, 6)
(1, 6) edge [ultra thick, red, dotted] (3, 8)
(3, 8) edge [ultra thick, red, dotted] (5, 6)
(14, 5) edge [ultra thick, red, dotted] (13.5, 4.5)
(10.5, 4.5) edge [ultra thick, red, dotted] (10, 5)
(10, 5) edge [ultra thick, red, dotted] (12, 7)
(12, 7) edge [ultra thick, red, dotted] (14, 5)
(11, 1) edge [ultra thick, red, dotted] (12, 2)
(12, 2) edge [ultra thick, red, dotted] (13, 1)
(13, 1) edge [ultra thick, red, dotted] (12, 0)
(12, 0) edge [ultra thick, red, dotted] (11, 1)
;
\end{tikzpicture}
\caption{The graph $H_{m,k}$ with columns labelled $c_1, c_2, \dots, c_k$ and $C$ indicated by the thick black rectangle. (Left) Resulting necessary broadcasts of strength 2 in columns $c_4$ and $c_{k-3}$ forced by vertices undominated in columns $c_6$ and $c_{k-5}$. (Right) Necessary broadcast $g'$ as described in Section \ref{subsec:necbroad}, the vertex undominated by $g'$ indicated by the green circle, and one of the five possible sub-broadcasts $g''$ which extend $g'$ to dominate the vertex undominated by $g'$ in column $c_{k-4}$.}
\label{fig:backtrack:notuniquefailedmaxregion}
\end{figure}
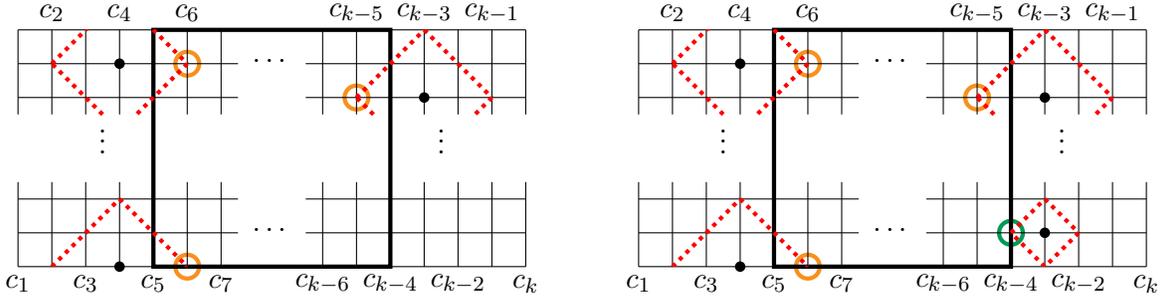

\subsection{Considering All Possible Sub-Cases}\label{lastsec}
When considering all possible sub-broadcasts $g$ and the necessary sub-broadcasts $g'$ they imply (see Section \ref{subsec:necbroad}), some sub-broadcasts may require that we consider all possible induced sub-broadcasts $g''$ which extend $g'$ to dominate $V(C)$.

Let $g'$ be the necessary broadcast implied by $g$ as described in Section \ref{subsec:necbroad}. Note $g'$ dominates $V(C)$ with the possible exception of vertices in columns $c_5$ and $c_{k-4}$. There are many possible ways a dominating $2$-limited broadcast could dominate the vertices in columns $c_5$ and $c_{k-4}$ undominated by $g'$. Let $ \mathcal{C}'$ be the collection of broadcasts formed by extending $g'$ to include vertices from $V(G_{m,k})\setminus \left[ V(c_5) \cup V(c_6) \cup\dots \cup V(c_{k-4}) \right] $ broadcasting at strength 0, 1, or strength 2 until every vertex in columns $c_5$ and $c_{k-4}$ which do not hear a broadcast under $g'$ is dominated. See Figure \ref{fig:backtrack:notuniquefailedmaxregion} (Right). Let $AllSubcases(H_{m,k},$ $g,$ $m_1,$ $m_2,$ $\dots,$ $m_k)$ return true if either $HasBroadcast(H_{m,k},$ $R',$ $cost(g')-1)$ or $InductiveArgument (H_{m,k},$ $R'',$ $cost(g''),$ $m_1,$ $m_2,$ $\dots,$ $m_k)$ returns true for every $g'' \in \mathcal{C'}$ (where $R''$ is the broadcast range of $g''$) and false otherwise.

Without loss of generality, we assume that, for each broadcast $g'' \in \mathcal{C}'$, $ForbiddenBroadcast(G_{m,k},$ $g'')$ is false. Additionally, for each $g'' \in \mathcal{C}'$, if the vertices $u$ and $v$ are added to the broadcast to dominate a vertex in column $c_5$, we may assume that the set of vertices dominated in column $c_5$ by $u$ is not a subset of the set of vertices in column $c_5$ dominated by $v$, else $u$ is redundant in terms of dominating the vertices of column $c_5$. The same argument applies to the vertices dominating the vertices in column $c_{k-4}$. 

\subsection{Algorithm to Prove Lower Bounds}\label{sec:middle}
Algorithm \ref{alg:thefinalfinal}: \textit{ProvedLowerBound} implements the battery of tests described in Section \ref{middlesec} through \ref{lastsec} to prove lower bounds for $C_m \square C_n$ and $P_m \square C_n$. The correctness of Algorithm \ref{alg:thefinalfinal} is proven by Theorem \ref{prop:lowerbounddone}.

\begin{algorithm}[htbp]
\SetKwInOut{Input}{Input}
\SetKwInOut{Output}{Output}
\underline{function ProvedLowerBound} $(H_{m,k},s,t,m_1, m_2,\dots, m_k)$\;
\Input{A graph $H_{m,k} = ( P_m \textnormal{ or } C_m) \square P_k $, where $k \geq 13$,  with columns labelled from left to right by $c_1, c_2,\dots, c_k$, the minimum $s$ and maximum $t$ possible costs of a sub-broadcast $g$ of $H_{m,k}$ induced by columns $c_5, c_6, \dots, c_{k-4}$, and $m_i$ (for each $i=1$ to $k$) as defined in Section \ref{sec:backtrack:actualmethods}.}
\Output{Value of the truth statement: ``lower bound proven.''}
Let $ \mathcal{C}$ be all  possible sub-broadcasts $g$  of costs $s$ to $t$ induced by the vertices of columns $c_5, c_6, \dots, c_{k-4}$\;
\ForEach{$g \in \mathcal{C}$}{
\textbf{if} $DoesNotDominate(H_{m,k},g)$ \textbf{then} \textbf{goto} line 12\;
\textbf{if} $ForbiddenBroadcast(H_{m,k},g)$ \textbf{then} \textbf{goto} line 12\;
Let $R$ be the broadcast range of $g$\;
\textbf{if} $HasBroadcast(H_{m,k}, R, cost(g)-1)$ \textbf{then} \textbf{goto} line 12\;
\textbf{if} $InductiveArgument(H_{m,k}, R, cost(g),m_1, m_2,\dots, m_k)$ \textbf{then} \textbf{goto} line 12\;
\textbf{if} $NecessaryBroadcast(H_{m,k},g,m_1,m_2,\dots,m_k)$ \textbf{then} \textbf{goto} line 12\;
\textbf{if} $AllSubcases(H_{m,k},g,m_1,m_2,\dots,m_k)$ \textbf{then} \textbf{goto} line 12\;
\Return False\;
\
}
\Return True\;
\caption{Routine to prove lower bound.}
\label{alg:thefinalfinal}
\end{algorithm}

\begin{thm}\label{prop:lowerbounddone}
    \sloppy Let $k \geq 13$ and $B(n)$ be a lower bound of $\gamma_{b,2}(G_{m,n})$ for all $n < n_0$ for some $n_0\geq k+3$.  Let $s = \gamma_{b,2}(H_{m,k-12})$ and $t=max \left\{\ell: \exists n \ge n_0 \textnormal{ s.t. } \frac{n\ell}{k-8} < B(n) \right\} $. If $ProvedLowerBound(H_{m,k},$ $s,$ $t,$ $m_1,$ $m_2,$ $\dots,$ $m_k)$ is true, then $B(n) \leq \gamma_{b,2}(G_{m,n})$ for all $n$. Here $m_i$ (for each $i=1$ to $k$) is defined as in Section \ref{sec:backtrack:actualmethods}.
\end{thm}

\begin{proof}
    \sloppy Assume the conditions of Theorem \ref{prop:lowerbounddone} and suppose $ProvedLowerBound(H_{m,k},$ $s,$ $t,$ $m_1,$ $m_2,$ $\dots,$ $m_k)$ is true. Let $f$ be an optimal $2$-limited dominating broadcast of $G_{m,n_0}$. Let $C$ be the subgraph of $G_{m,n_0}$ induced by the vertices appearing in a minimum cost set (with respect to $f$) of $r=k-8$ consecutive columns of $G_{m,n_0}$. Let $H_{m,k}$ be the subgraph of $G_{m,n_0}$ centred on $C$ with columns labelled $c_1, c_2, \dots, c_k$ and let $C'$ be the subgraph of $H_{m,k}$ induced by columns $c_7, c_8, \dots, c_{k-6}$. Note the vertices of $C$ are in columns $c_5, c_6,\dots , c_{k-4}$ of $H_{m,k}$ and $C' = H_{m,k-12}$. The sub-broadcast $g$ of $f$ induced by $V(C)$ must dominate $C'$ (Observation \ref{observation2}). As $s = \gamma_{b,2}(C')$, we have that $cost(g) \geq s$. By our choice of $C$,
    \begin{equation*}
    \begin{aligned}
        cost(f) \geq \frac{n \cdot cost(g)}{r} = \frac{n \cdot cost(g)}{k-8}.
    \end{aligned}
    \end{equation*}
    If $cost(g) > t$ then, by the definition of $t$, $cost(f) \geq B(n)$ for all $n \geq n_0$. Thus, $cost(g) \le t$. As $\mathcal{C}$ defined on line 2 is the set of all possible sub-broadcasts of cost between $s$ and $t$, inclusive, induced by the vertices of columns $c_5, c_6, \dots, c_{k-4}$, $g \in \mathcal{C}$. As $ProvedLowerBound$ returned true, one of the function calls on lines 4 through 10 returned true for $g$. From the results in Sections \ref{middlesec} through \ref{lastsec}, this proves the claim.
\end{proof}

\section{Results}\label{sec:bactrack:app of methods}
Our implementation includes a canonicity test for $ P_{m} \square P_{n} $ so as to only consider the set $ \mathcal{C}$ of all possible sub-broadcasts induced by the vertices of some set of $r$ consecutive columns  with costs between $s$ and $t$, inclusive, up to isomorphism. That is, for each pair of broadcasts $g, g^* \in \mathcal{C}$, there is no group action on $P_m \square P_n$ which defines an automorphism between $g$ and $g^*$. This test was done by checking that each broadcast (when expressed as a sequence)  was the maximum lexicographically when compared to all broadcasts isomorphic to it. When adapting the code to work on $ C_{m} \square C_{n} $, we did not update the canonicity test to reduce the number of cases up to isomorphism on $ C_{m} \square P_{n} $ from $ P_{m} \square P_{n} $. Fortunately, this redundancy was acceptable in terms of run time. The number of induced sub-broadcasts (i.e. $\lvert \mathcal{C} \rvert $) have been verified by P\'{o}lya's Theorem (see \cite[Theorem 14.3.3]{brualdi2010introductory}).

Our implementation of $ProvedLowerBound$ has allowed us to prove Theorems \ref{thm:backtrack:C3} through \ref{thm:backtrack:PC5}, and their respective corollaries. For each theorem, we include a table which summarizes the number of broadcasts rejected at each step of the algorithm per considered cost.  Steps with zero cases are omitted. Additionally, as $AllSubcases$ considers all possible induced sub-broadcasts of a given case, the total number of cases considered will be at least $ \left| \mathcal{C} \right| $.

Our implementation is written in C++ and available here \cite{Slobodin_C_implementation_of_2022}. All ILP calls are run with a Gurobi solver \cite{gurobi}. All computations in this section were run on Slobodin's 2021 16GB MacBook Pro with an Apple M1 Pro processor.

\begin{thm}\label{thm:backtrack:C3}
For $n \geq 3$, $\gamma_{b,2} \left( C_{3} \square C_{n} \right)  =  \left\lceil \frac{ 2n }{ 3  }  \right\rceil$.

\end{thm}
\begin{proof}
Theorem 6 of \cite{slobodinfinal} proves that $ \gamma_{b,2} \left( C_{3} \square C_{n} \right) \leq \left\lceil 2n/3  \right\rceil  $. Fix $r=6$ and let $k =14 = r+8$. By computation, we know that the upper bound is optimal for all $3 \leq n \leq 16 = k+2$. Given the upper bound, for $1 \leq i \leq 14 = k$, $m_i$ is defined as follows:
\begin{equation*}
\begin{aligned}
( m_1, m_2,\dots, m_{14} ) = (  1, 2, 2, 3, 4, 4, 5, 6, 6, 7, 8, 8, 9, 10).
\end{aligned}
\end{equation*}
\sloppy As $r-4 = 2$ and $ \gamma_{b,2} \left( C_3 \square P_2 \right) = 2$, set $s=2$. Set $t = 3$. Observe that, for $n = 17$, 
\begin{equation*}
\begin{aligned}
\left\lceil \frac{3n}{6}  \right\rceil = 9 < 12 = \left\lceil \frac{2n}{3}  \right\rceil = B(n),
\end{aligned}
\end{equation*}
thus $t \geq 3$. If $t>3$, then there exists an $n > 16$ such that
\begin{equation*}
\begin{aligned}
\left\lceil \frac{4n}{6}  \right\rceil \leq \frac{tn }{6} < B(n) = \left\lceil \frac{4n}{6}  \right\rceil,
\end{aligned}
\end{equation*}
which is a contradiction. As $ProvedLowerBound\left( C_3 \square P_{14}, 2, 3, m_1, m_2,\dots, m_{14}\right)$ is true, the result follows. 
\end{proof}

Running \textit{ProvedLowerBound} for the above values took less than one second. 
\begin{table}[htbp]
\centering
\begin{tabular}{l|c|c|}
& Cost 2 & Cost 3  \\ \hline
$ \left| \mathcal{C} \right|$  & 54 & 302  \\\hline
$DoesNotDominate$\hfill & 48 & 231\\\hline
$ForbiddenBroadcast$\hfill & 4 & 45\\\hline
$InductiveArgument$\hfill & 0 & 12 \\\hline
$NecessaryBroadcast+HasBroadcast$ & 0 & 8\\ \hline
$NecessaryBroadcast+InductiveArgument$ & 2 & 3\\ \hline
$AllSubcases+HasBroadcast$ & 0 & 45\\ \hline
$AllSubcases+InductiveArgument$ & 0 & 63\\
\end{tabular}
\caption{Cases considered in the proof of Theorem \ref{thm:backtrack:C3}.}
\label{table:backtrackC3}
\end{table}

\begin{cor}\label{cor:backtrack:C3Pn}
For $m \geq 3$, $\gamma_{b,2} \left( P_{m} \square C_{3}  \right) = \left\lceil \frac{ 2m }{ 3 }  \right\rceil$.
\end{cor}
\begin{proof}
The bound is easily verified by computation for $3 \leq m \leq 22$. Theorem 5 of \cite{slobodinfinal} proves that $\gamma_{b,2} \left( P_{m} \square C_{3}  \right) \leq \left\lceil 2n/3  \right\rceil  $ for all $m \geq 23$. As any $2$-limited dominating broadcast on $ P_{m} \square C_{3} $ is a $2$-limited dominating broadcast on $ C_{m} \square C_{3} $, $ \gamma_{b,2} \left( C_{3} \square C_{m} \right) \leq \gamma_{b,2} \left( P_{m} \square C_{3} \right) $. The result follows from Theorem \ref{thm:backtrack:C3}.
\end{proof}

\begin{thm}\label{thm:backtrack:C4}
For $n \geq 4$, $ \gamma_{b,2} \left( C_{4} \square C_{n} \right)  =  4 \left\lfloor \frac{ n }{ 6 }  \right\rfloor +  \begin{cases}
0 & \textnormal{for } n \equiv 0 \pmod{ 6}, \\
2 & \textnormal{for } n \equiv 1 \textnormal{ or } 2 \pmod{ 6}, \\
3 & \textnormal{for } n \equiv 3 \textnormal{ or } 4 \pmod{ 6}, \textnormal{ and}   \\
4 & \textnormal{for } n \equiv 5 \pmod{ 6 }. \\
\end{cases}$
\end{thm}
\begin{proof}
Theorem 6 of \cite{slobodinfinal} proves that $ \gamma_{b,2} \left( C_{4} \square C_{n} \right) $ is less than or equal to the stated value. Fix $r=6$ and let $k =14 = r+8$. By computation, we know that the stated value is optimal for all $3 \leq n \leq 16 = k+2$. Given the upper bound, for $1 \leq i \leq 14 = k$,  $m_i$ is defined as follows:
\begin{equation*}
\begin{aligned}
( m_1, m_2,\dots, m_{14} ) = (2, 2, 3, 3, 4, 4, 6, 6, 7, 7, 8, 8, 10, 10).
\end{aligned}
\end{equation*}
As $r-4 = 2$ and $ \gamma_{b,2} \left( C_4 \square P_2 \right) = 2$, set $s=2$. Let $n_6$ be the least residue of $n$ modulo $6$ and let $c(n_6)$ be the constant in the upper bound dependent upon $n_6$. Set $t=4$. Observe that, for $n = 19$, 
\begin{equation*}
\begin{aligned}
\left\lceil \frac{4n}{6}  \right\rceil = 13 < 14=4 \left\lfloor \frac{n}{6}  \right\rfloor + 2=4 \left\lfloor \frac{n}{6}  \right\rfloor +n_6 = B(n),
\end{aligned}
\end{equation*}
thus $t \geq 4$.  If $t>4$, then there exists an $n > 16$ such that
\begin{equation*}
\begin{aligned}
\frac{5n}{6} \leq \frac{tn }{6} < B(n) =4 \left\lfloor \frac{n}{6}  \right\rfloor + c(n_6) = \frac{4n}{6} - \frac{4n_6}{6} + c(n_6) \leq \frac{4n}{6} + \frac{4}{3} \Rightarrow n < 8,
\end{aligned}
\end{equation*}
which is a contradiction. As $ProvedLowerBound\left( C_4 \square P_{14}, 2, 4, m_1, m_2,\dots, m_{14}\right)$ is true, the result follows. 
\end{proof}

Running \textit{ProvedLowerBound} for the above values took less than one second.
\begin{table}[htbp]
\centering
\begin{tabular}{l|c|c|c|}
& Cost 2 & Cost 3 & Cost 4 \\ \hline
$ \left| \mathcal{C} \right|$ & 84 & 644 & 4,302  \\\hline
$DoesNotDominate$ & 83 & 610 & 3,770\\\hline
$ForbiddenBroadcast$ & 0 & 16 & 378\\\hline
$InductiveArgument$ & 0 & 0 & 98 \\\hline
$NecessaryBroadcast + HasBroadcast$ & 1 & 16 & 33\\\hline
$NecessaryBroadcast+InductiveArgument$ &  0 &  2 &  20\\\hline
$AllSubcases+HasBroadcast$ &  0 &  0 &  2\\\hline
$AllSubcases+InductiveArgument$ &  0 &  0 &   30\\
\end{tabular}
\caption{Cases considered in the proof of Theorem \ref{thm:backtrack:C4}.}
\label{table:backtrackC4}
\end{table}

\begin{cor}\label{cor:backtrack:PnC4}
For $m \geq 3$,
$\gamma_{b,2} \left( P_m \square C_4 \right) = 4 \left\lfloor \frac{ m }{ 6 }  \right\rfloor +  
\begin{cases}
0 \textnormal{ or } 1 & \textnormal{for } m \equiv 0 \pmod{ 6 }, \\
2 & \textnormal{for } m \equiv 1 \textnormal{ or } 2 \pmod{ 6 },\\
3 & \textnormal{for } m \equiv 3 \pmod{ 6 },  \\
3 \textnormal{ or } 4  & \textnormal{for } m \equiv 4 \pmod{ 6 }, \textnormal{ and}   \\
4 & \textnormal{for } m \equiv 5  \pmod{ 6 }. \\
\end{cases}$
\end{cor}
\begin{proof}
The bound is easily verified by computation for $3 \leq m \leq 22$. Theorem 5 of \cite{slobodinfinal} proves that $\gamma_{b,2} \left( P_{m} \square C_{4}  \right)$ is less than or equal to the bound in the corollary statement for $m \geq 23$. As any $2$-limited dominating broadcast on $ P_{m} \square C_{4} $ is a $2$-limited dominating broadcast on $ C_{m} \square C_{4} $, $ \gamma_{b,2} \left( C_{4} \square C_{m} \right) \leq \gamma_{b,2} \left( P_{m} \square C_{4} \right) $. The result follows from Theorem \ref{thm:backtrack:C4}.
\end{proof}

\begin{thm}\label{thm:backtrack:C5}
For $n \geq 5$, $\gamma_{b,2} \left( C_{5} \square C_{n} \right) = n$.
\end{thm}
\begin{proof}
Theorem 6 of \cite{slobodinfinal} proves that $ \gamma_{b,2} \left( C_{5} \square C_{n} \right) \leq n$. Fix $r=8$ and let $k =16 = r+8$. By computation, we know that the upper bound is optimal for all $3 \leq n \leq 18 = k+2$. Given the upper bound, for $1 \leq i \leq 16 = k$, $m_i$ is defined as follows: 
\begin{equation*}
\begin{aligned}
( m_1, m_2,\dots, m_{16} ) = ( 1,2,3,4,5,6,7,8,9,10,11,12,13,14,15,16).
\end{aligned}
\end{equation*}
As $r-4 = 4$ and $ \gamma_{b,2} \left( C_5 \square P_4 \right) = 5$, set $s=5$. Set $t=7$. Observe that, for $n = 19$, 
\begin{equation*}
\begin{aligned}
\left\lceil \frac{7n}{8}  \right\rceil = 17 < 19= n = B(n),
\end{aligned}
\end{equation*}
thus $t \geq 7$. If $t>7$, then there exists an $n > 18$ such that
\begin{equation*}
\begin{aligned}
\frac{8n}{8}  \leq \frac{tn }{8} < B(n) = n,
\end{aligned}
\end{equation*}
which is a contradiction. As $ProvedLowerBound\left( C_5 \square P_{16}, 5, 7, m_1, m_2,\dots, m_{16}\right)$ is true, the result follows. 
\end{proof}

Running \textit{ProvedLowerBound} for the above values took less than one minute.  
\begin{table}[htbp]
\centering
\begin{tabular}{l|c|c|c|}
& Cost 5 & Cost 6 & Cost 7 \\ \hline
$ \left| \mathcal{C} \right|$ & 264,148 & 1,925,104 & 12,162,548  \\\hline
$DoesNotDominate$ & 264,115 & 1,922,880 & 12,103,722\\\hline
$ForbiddenBroadcast$ & 8 & 1,423 & 48,899\\\hline
$HasBroadcast$ & 0 & 161 & 5,198\\\hline
$InductiveArgument$ & 25 & 632 & 4,696 \\\hline
$NecessaryBroadcast$ & \multirow{2}{*}{0} & \multirow{2}*{5} & \multirow{2}*{27}\\
$+InductiveArgument$ & & &\\\hline
$AllSubcases+HasBroadcast$ &  0 &  27 &  0\\\hline
$AllSubcases$ & \multirow{2}*{0} & \multirow{2}*{48} &  \multirow{2}*{30}\\
$+InductiveArgument$ & & &\\       
\end{tabular}
\caption{Cases considered in the proof of Theorem \ref{thm:backtrack:C5}.}
\label{table:backtrackC5}
\end{table}

\begin{cor}\label{cor:backtrack:CP5}
For $m \geq 3$, $m \leq \gamma_{b,2} \left( P_{m} \square C_{5} \right) \leq m+1$.
\end{cor}
\begin{proof}
Theorem 5 of \cite{slobodinfinal} proves that $ \gamma_{b,2}(P_m \square C_5) \leq m+1$. The lower bound is easily verified by computation for $3 \leq m \leq 4$. As any $2$-limited dominating broadcast on $ P_{m} \square C_{5} $ is a $2$-limited dominating broadcast on $ C_{m} \square C_{5} $, $ \gamma_{b,2} \left( C_{5} \square C_{m} \right) \leq \gamma_{b,2} \left( P_{m} \square C_{5} \right) $. The result follows from Theorem \ref{thm:backtrack:C5}.
\end{proof}

\begin{thm}\label{thm:backtrack:C6}
For $n \geq 6$, $\gamma_{b,2} \left( C_{6} \square C_{n} \right)  =  n +  \begin{cases}
0 & \textnormal{for } n \equiv 0 \pmod{ 4} \textnormal{ and}  \\
1 & \textnormal{for } n \equiv 1, 2, \textnormal{ or } 3 \pmod{ 4}. \\
\end{cases}$
\end{thm}
\begin{proof}
Theorem 6 of \cite{slobodinfinal} proves that $ \gamma_{b,2} \left( C_{6} \square C_{n} \right) $ is less than or equal to the stated value. Fix $r=8$ and let $k =16 = r+8$. By computation, we know that the stated value is optimal for all $3 \leq n \leq 18 = k+2$. Given the upper bound, for $1 \leq i \leq 16 = k$, $m_i$ is defined as follows:
\begin{equation*}
\begin{aligned}
( m_1, m_2,\dots, m_{16} ) = (2, 3, 4, 4, 6, 7, 8, 8, 10, 11, 12, 12, 14, 15, 16, 16).
\end{aligned}
\end{equation*}
As $r-4 = 4$ and $ \gamma_{b,2} \left( C_6 \square P_4 \right) = 5$, set $s=5$. Let $n_4$ be the least residue of $n$ modulo $4$ and let $c(n_4)$ be the constant in the upper bound dependent upon $n_4$. Set $t=8$. Observe that, for $n = 19$, 
\begin{equation*}
\begin{aligned}
\left\lceil \frac{8n}{8}  \right\rceil = 19 < 20=n+1 = n+c(n_4) = B(n),
\end{aligned}
\end{equation*}
thus $t \geq 8$. If $t>8$, then there exists an $n > 18$ such that
\begin{equation*}
\begin{aligned}
\frac{9n}{8} \leq \frac{tn}{8} < B(n) = n + c(n_4) \leq n+1 \Rightarrow n < 8
\end{aligned}
\end{equation*}
which is a contradiction. As $ProvedLowerBound\left( C_6 \square P_{16}, 5, 8, m_1, m_2,\dots, m_{16}\right)$ is true, the result follows. 
\end{proof}

Running \textit{ProvedLowerBound} for the above values took less than seven minutes. 
\begin{table}[htbp]
\centering
\begin{tabular}{l|c|c|c|c|}
& Cost 5 & Cost 6 & Cost 7 & Cost 8\\ \hline
$ \left| \mathcal{C} \right|$ & 635,628 & 5,506,384 & 41,289,876 & 273,548,430 \\\hline
$DoesNotDominate$ & 635,625 & 5,506,080 & 41,277,225 & 273,227,125\\\hline
$ForbiddenBroadcast$ & 0 & 138 & 9,204 & 278,760\\\hline
$HasBroadcast$ & 0 & 21 & 1,368 & 31,477\\\hline
$InductiveArgument$ & 0 & 67 & 1,698 & 10,361 \\\hline
$NecessaryBroadcast$ & \multirow{2}*{3} & \multirow{2}*{78} & \multirow{2}*{330} &  \multirow{2}*{563}\\
$+$ $HasBroadcast$ & & & &\\\hline 
$NecessaryBroadcast$ & \multirow{2}{*}{0} & \multirow{2}*{0} & \multirow{2}*{39} &  \multirow{2}*{102}\\
$+InductiveArgument$ & &  & &\\\hline
$AllSubcases$ & \multirow{2}*{0} & \multirow{2}*{0} & \multirow{2}*{1,262} &  \multirow{2}*{5914}\\
$+HasBroadcast$ & & & &\\\hline
$AllSubcases$ & \multirow{2}*{0} & \multirow{2}*{0} &  \multirow{2}*{78} &  \multirow{2}*{204}\\
$+InductiveArgument$ & & & &\\
\end{tabular}
\caption{Cases considered in the proof of Theorem \ref{thm:backtrack:C6}.}
\label{table:backtrackC6}
\end{table}

\begin{cor}\label{cor:backtrack:PnC6}
For $m \geq 3$, $\gamma_{b,2} \left( P_m \square C_6 \right) =  m + \begin{cases}
0 \textnormal{ or } 1 & \textnormal{for } m \equiv 0 \pmod{ 4} \textnormal{ and}  \\
1 & \textnormal{for } m \equiv 1, 2, \textnormal{ or } 3 \pmod{ 4}. \\
\end{cases}$
\end{cor}
\begin{proof}
Theorem 5 of \cite{slobodinfinal} proves that $\gamma_{b,2}(P_m \square C_6)\leq m+1$. The lower bound is easily verified by computation for $3 \leq m \leq 5$. As any $2$-limited dominating broadcast on $ P_{m} \square C_{6} $ is a $2$-limited dominating broadcast on $ C_{m} \square C_{6} $, $ \gamma_{b,2} \left( C_{6} \square C_{m} \right) \leq \gamma_{b,2} \left( P_{m} \square C_{6} \right) $. The result follows from Theorem \ref{thm:backtrack:C6}.
\end{proof}

\begin{thm}\label{thm:backtrack:PC4}
For $n \geq 3$, $\gamma_{b,2} \left( P_4 \square C_n \right) = 8 \left\lfloor \frac{ n }{ 10 }  \right\rfloor + c(n_{10})$ 
where $c(n_{10}) $ is dependent upon the least residue $n_{10}$ of $n$ modulo $10$ and given in Table \ref{table:endtermPC4}.
\begin{table}[htbp]
\centering
\begin{tabular}{c|cccccccccccc}
& \multicolumn{12}{c}{Least residue $n_{10}$ of $n$ modulo $10$} \tabularnewline
$n_{10}$ & $0$ & $1$ & $2$ & $3$ & $4$ & $5$ & $6$ &$ 7$ &$ 8$ &$ 9$  \\ \hline
$c(n_{10})$: & $0$ & $2$ & $2$ & $3$ & $4$ & $5$ & $5$ & $6$ & $7$ & $8$ \\
\end{tabular}
\caption{Values of $c(n_{10})$ for  $ \gamma_{b,2} \left( P_4 \square C_{n} \right) $ stated in Theorem \ref{thm:backtrack:PC4}.}
\label{table:endtermPC4}
\end{table}
\end{thm}

\begin{proof}
Theorem 3 of \cite{slobodinfinal} proves that $ \gamma_{b,2} \left( P_{4} \square C_{n} \right) $ is less than or equal to the stated value. Fix $r=10$ and let $k=18=r+8$. By computation, we know that the stated value is optimal for all $3 \leq n \leq 20 = k+2$.  Given the upper bound, for $1 \leq i \leq 18 = k$,  $m_i$ is defined as follows:
\begin{equation*}
\begin{aligned}
( m_1, m_2,\dots, m_{18} ) = (2, 2, 3, 4, 5, 5, 6, 7, 8, 8, 10, 10, 11, 12, 13, 13, 14, 15).
\end{aligned}
\end{equation*}
As $r-4 = 6$ and $ \gamma_{b,2} \left( P_4 \square P_6 \right) = 6$, set $s=6$. Let $n_{10}$ be the least residue of $n$ modulo $10$ and let $c(n_{10})$ be the constant in the upper bound dependent upon $n_{10}$. Set $t=8$. Observe that, for $n = 21$, 
\begin{equation*}
\begin{aligned}
    \left\lceil \frac{8n}{10}  \right\rceil = 17 < 18=8 \left\lfloor \frac{n}{10}  \right\rfloor + 2 =8 \left\lfloor \frac{n}{10}  \right\rfloor + c(n_{10}) = B(n),
\end{aligned}
\end{equation*}
thus $t \geq 8$. If $t>8$, then there exists an $n > 20$ such that
\begin{equation*}
\begin{aligned}
\frac{9n}{10} \leq \frac{t n}{10}  <B(n) = 8 \left\lfloor \frac{n}{10}  \right\rfloor + c(n_{10}) = \frac{8n}{10} - \frac{8n_{10}}{10} + c(n_{10}) \leq \frac{8n}{10} + \frac{12}{10} \Rightarrow n <12,
\end{aligned}
\end{equation*}
which is a contradiction. As $ProvedLowerBound\left( P_4 \square P_{18}, 6, 8, m_1, m_2,\dots, m_{18}\right)$ is true, the result follows. 
\end{proof}

Running \textit{ProvedLowerBound} for the above values took less than 30 seconds. 
\begin{table}[htbp]
\centering
\begin{tabular}{l|c|c|c|}
& Cost 6 & Cost 7 & Cost 8 \\ \hline
$ \left| \mathcal{C} \right|$ & 1,922,800 & 12,154,870 & 67,920,535  \\\hline
$DoesNotDominate$ & 1,922,790 & 12,153,957 & 67,886,561\\\hline
$ForbiddenBroadcast$& 1 & 546 & 27,525\\\hline
$HasBroadcast$ & 0 & 136 & 4,944\\\hline
$InductiveArgument$ & 7 & 215 & 1,472 \\\hline
$NecessaryBroadcast+HasBroadcast$ &  1 &  8 &  16\\\hline
$NecessaryBroadcast$ & \multirow{2}{*}{0} & \multirow{2}*{6} & \multirow{2}*{12}\\ 
$+InductiveArgument$ & & &\\\hline
$AllSubcases+HasBroadcast$ &  3 &  8 &  61\\ \hline
$AllSubcases$ & \multirow{2}*{13} & \multirow{2}*{28} &  \multirow{2}*{43}\\
$+InductiveArgument$ & & &\\
\end{tabular}
\caption{Cases considered in the proof of Theorem \ref{thm:backtrack:PC4}.}
\label{table:backtrack:PC4}
\end{table}

\begin{cor}\label{cor:backtrack:P4Pn}
(a) For $n \geq 4$, $8 \left\lfloor \frac{ n }{ 10 }  \right\rfloor + c(n_{10}) \leq \gamma_{b,2} \left( P_4 \square P_n \right)$ where $c(n_{10}) $ is dependent upon the least residue $n_{10}$ of $n$ modulo $10$ and given in Table \ref{table:endtermPC4}. (b) The lower bound stated in (a) gives optimal values for $ \gamma_{b,2}(P_4 \square P_n)$ for all $n \equiv 1, 4, 5, \textnormal{ and } 9 \pmod{ 10 }$
\end{cor}
\begin{proof}
(a) As any $2$-limited dominating broadcast on $ P_{4} \square P_{n} $ is a $2$-limited dominating broadcast on $ P_{4} \square C_{n} $, $ \gamma_{b,2} \left( P_{4} \square C_{n} \right) \leq \gamma_{b,2} \left( P_{4} \square P_{n} \right) $. The lower bound follows from Theorem \ref{thm:backtrack:PC4}. (b) Theorem 1 of \cite{slobodinfinal} proves that $\gamma_{b,2} \left( P_4 \square P_n \right) \leq 8 \left\lfloor \frac{ n }{ 10 }  \right\rfloor + d(n_{10})$ where  $d(n_{10})$ is dependent upon the least residue $n_{10}$ of $n$ modulo $10$ and given in Table \ref{table:endtermP4Pn}. The result follows.
\end{proof}

\begin{table}[htbp]
\centering
\begin{tabular}{c|cccccccccccc}
& \multicolumn{12}{c}{Least residue $n_{10}$ of $n$ modulo $10$} \tabularnewline
$n_{10}$ & $0$ & $1$ & $2$ & $3$ & $4$ & $5$ & $6$ &$ 7$ &$ 8$ &$ 9$  \\ \hline
$d(n_{10})$ & $1$ & $2$ & $3$ & $4$ & $4$ & $5$ & $6$ & $7$ & $8$ & $8$ \\
\end{tabular}
\caption{Values of $d(n_{10})$ for  $ \gamma_{b,2} \left( P_4 \square P_{n} \right) $ stated in Theorem 3.1 of \cite{slobodinfinal}.}
\label{table:endtermP4Pn}
\end{table}

\begin{thm}\label{thm:backtrack:PC5}
For $n \geq 3$, $\gamma_{b,2} \left( P_{5} \square C_{n} \right)  =  n +  \begin{cases}
0 & \textnormal{for } n \equiv 0 \pmod{ 2} \textnormal{ and} \\
1 & \textnormal{for } n \equiv 1 \pmod{ 2}. \\
\end{cases}$
\end{thm}
\begin{proof}
Theorem 3 of \cite{slobodinfinal} proves that $ \gamma_{b,2} \left( P_{5} \square C_{n} \right) $ is less than or equal to stated value. Fix $r=11$ and let $k=19=r+8$. By computation, we know that the stated value is optimal for all $3 \leq n \leq 21 = k+2$. Given the upper bound, for $1 \leq i \leq 19 = k$, $m_i$ is defined as follows:
\begin{equation*}
\begin{aligned}
( m_1, m_2,\dots, m_{19} ) = (2, 2, 4, 4, 6, 6, 8, 8, 10, 10, 12, 12, 14, 14, 16, 16, 18, 18, 20).
\end{aligned}
\end{equation*}
As $r-4 = 7$ and $ \gamma_{b,2} \left( P_5 \square P_7 \right) = 8$, set $s=8$. Let $n_2$ be the least residue of $n$ modulo $2$ and let $c(n_2)$ be the constant in the upper bound dependent upon $n_2$. Set $t=11$. Observe that, for $n = 23$, 
\begin{equation*}
\begin{aligned}
    \left\lceil \frac{11n}{11}  \right\rceil = 23 < 24= n + 1= n + c(n_2)= B(n),
\end{aligned}
\end{equation*}
thus $t \geq 11$. If $t>11$, then there exists an $n > 21$ such that
\begin{equation*}
\begin{aligned}
\frac{12n}{11} \leq \frac{t n}{11} < B(n) = n + c(n_2) \leq n+1 \Rightarrow n <11,
\end{aligned}
\end{equation*}
which is a contradiction. As $ProvedLowerBound\left( P_5 \square P_{19}, 8, 11, m_1, m_2,\dots, m_{19}\right)$ is true, the result follows. 
\end{proof}

Running \textit{ProvedLowerBound} for the above values took less than 30 minutes. 
\begin{table}[htbp]
\centering
\begin{tabular}{l|c|c|c|c|}
& Cost 8 & Cost 9 & Cost 10 & Cost 11\\ \hline
$ \left| \mathcal{C} \right|$ & 777,158,275 & 5,239,827,968 & 32,027,967,253 & 179,128,860,188 \\\hline
$DoesNotDominate$ & 777,158,269 & 5,239,826,944 & 32,027,887,652 & 179,125,748,233\\\hline
$ForbiddenBroadcast$ & 0 & 514 & 61,253 & 2,782,915\\\hline
$HasBroadcast$ & 0 & 213 & 14,112 & 311,874\\\hline
$InductiveArgument$ & 5 & 287 & 4,208 & 17,127 \\\hline
$NecessaryBroadcast$ & \multirow{2}*{1} & \multirow{2}*{7} & \multirow{2}*{14} &  \multirow{2}*{21}\\
$+$ $HasBroadcast$ & & & &\\\hline 
$NecessaryBroadcast$ & \multirow{2}{*}{0} & \multirow{2}*{3} & \multirow{2}*{12} &  \multirow{2}*{14}\\
$+InductiveArgument$ & &  & &\\\hline
$AllSubcases$ & \multirow{2}*{0} & \multirow{2}*{0} & \multirow{2}*{1} &  \multirow{2}*{25}\\
$+HasBroadcast$ & & & &\\\hline
$AllSubcases$ & \multirow{2}*{0} & \multirow{2}*{0} &  \multirow{2}*{15} &  \multirow{2}*{26}\\
$+InductiveArgument$ & & & &\\
\end{tabular}
\caption{Cases considered in the proof of Theorem \ref{thm:backtrack:PC5}.}
\label{table:backtrack:PC5}
\end{table}

\begin{cor}\label{cor:backtrack:P5}
For $n \geq 5$, $\gamma_{b,2} \left( P_{5} \square P_{n} \right) = n + \begin{cases}
0 \textnormal{ or } 1 & \textnormal{for } n \equiv 0  \pmod{ 2} \textnormal{ and} \\
1 & \textnormal{for } n \equiv 1 \pmod{ 2}.\\
\end{cases}$
\end{cor}
\begin{proof}
Theorem 1 of \cite{slobodinfinal} proves that  $ \gamma_{b,2} \left( P_{5} \square P_{n} \right) \leq n+1$. As any $2$-limited dominating broadcast on $ P_{5} \square P_{n} $ is a $2$-limited dominating broadcast on $ P_{5} \square C_{n} $, $ \gamma_{b,2} \left( P_{5} \square C_{n} \right) \leq \gamma_{b,2} \left( P_{5} \square P_{n} \right) $. The result follows from Theorem \ref{thm:backtrack:PC5}.
\end{proof}

\section{Future Work}\label{sec:conclusion}
This paper presents a method for computationally proving lower bounds for the $2$-limited broadcast domination of the Cartesian product of two paths, a path and a cycle, and two cycles. Exact values for the $2$-limited broadcast domination number of $C_m\square C_n$ for $3 \leq m \leq 6$ and all $n\geq m$, $P_m \square C_3$ for all $m \geq 3$, and $P_m \square C_n$ for $4\leq m \leq 5$ and all $n \geq m$ have been found, as have periodically optimal values for $P_m \square C_4$ and $P_m \square C_6$ for $m \geq 3$, $P_4 \square P_n$ for $n \geq 4$, and $P_5 \square P_n$ for $n \geq 5$. Our method can likely be extended to other graphs and $k$-limited broadcast domination for $k > 2$. We note the follow rather natural questions.

\begin{problem}
Can this method be optimized further to prove bounds on larger graphs or graph other than the Cartesian product of two paths, a path and a cycle, and two cycles?
\end{problem}

\begin{problem}
Can this method be altered to prove bounds for the $k$-limited broadcast domination number on the Cartesian product of two paths, a path and a cycle, and two cycles?
\end{problem}

\textbf{Note:} Using the methods described in this paper, and an improved backtracking technique, we have also proven, for $n \geq 8$,
$$\gamma_{b,2} \left( C_8 \square C_n \right) = 8 \left\lfloor \frac{ n }{ 6 }  \right\rfloor +  
\begin{cases}
0 & \textnormal{for } n \equiv 0 \pmod{ 6 }, \\
2 & \textnormal{for } n \equiv 1  \pmod{ 6 },\\
4 & \textnormal{for } n \equiv 2 \pmod{ 6 },  \\
6 & \textnormal{for } n \equiv 3 \textnormal{ or } 4  \pmod{ 6 }, \textnormal{ and}   \\
8 & \textnormal{for } n \equiv 5  \pmod{ 6 }. \\
\end{cases}$$
This computation took one year and considered over 223 trillion cases. For each proof in this paper, we used a backtracking algorithm to construct the set $ \mathcal{C}$ of all possible sub-broadcasts. In our improved backtracking algorithm, we forbid the addition of any ForbiddenBroadcast (see Section \ref{sec:forbid}). The number of cases produced by this improved backtrack cannot be verified by P\'{o}lya's Theorem (see \cite[Theorem 14.3.3]{brualdi2010introductory}). As such, these results will be reported elsewhere with an updated methodology and justification.

\acknowledgements
\label{sec:ack}
All computations for this paper we run on machines purchased by Myrvold using NSERC funding. We
would also like to thank the referees for their thoughtful and instructive comments

\nocite{*}
\bibliographystyle{abbrvnat}
\bibliography{mybib}
\label{sec:biblio}

\end{document}